\documentclass[12pt,oneside]{article}
\usepackage{times,amsmath,amsbsy,amssymb,amscd,mathrsfs,graphicx}
\usepackage{amssymb}
\usepackage{amsmath,bm}
\usepackage{amsmath,amsthm}

\usepackage{subcaption}
\usepackage{framed}
\usepackage{xcolor,color}
\usepackage[font=small,labelfont=md,textfont=it]{caption}
\usepackage{floatrow,float}
\usepackage[titletoc, title]{appendix}
\usepackage[colorlinks,linkcolor=blue,citecolor=blue]{hyperref}
\usepackage{longtable}
\usepackage{pgfplots}
\usepackage{diagbox}
\usepackage{booktabs,makecell,multirow}
\usepackage[capitalise, nosort]{cleveref}
\usepackage{algorithm}
\usepackage{algorithmic}
\usepackage{amsmath,bm}
\usepackage{cases}
\usepackage{geometry}
\usepackage{extarrows}
\usepackage{tcolorbox}
\usepackage{bbm}
\usepackage{syntonly}
\usepackage[figuresright]{rotating}
 \usepackage{epstopdf}

\crefname{equation}{}{}
\crefname{lem}{Lemma}{Lemmas}
\crefname{thm}{Theorem}{Theorems}

\newcommand{\whk}[0]{\widehat{\xi}_k}
\newcommand{\proj}[0]{ {\bf proj}}

\newcommand{\dd}{\,{\rm d}}

\newcommand{\R}{\,{\mathbb R}}

\newcommand{\brk}[0]{\bar{\xi}_k}
\newcommand{\brxk}[0]{\bar{x}_k}
\newcommand{\dual}[1]{\left\langle {#1} \right\rangle}
\newcommand{\rg}[0]{ {{\bf  range}\,}}
\newcommand{\diam}[1]{ {{\bf  diam}\left({#1}\right)}}

\newcommand{\argmin}[0]{ {\mathop{{\rm  argmin}}\,}}

\newcommand{\st}[0]{ {{\rm  s.t.}\,}}
\newcommand{\argmax}[0]{ {\mathop{{\rm  argmax}}\,}}

\newcommand{\nm}[1]{\left\lVert {#1} \right\rVert}
\newcommand{\snm}[1]{\left\lvert {#1} \right\rvert}

\newcommand{\ssnm}[1]
{
	\left\vert\kern-0.25ex
	\left\vert\kern-0.25ex
	\left\vert
	{#1}
	\right\vert\kern-0.25ex
	\right\vert\kern-0.25ex
	\right\vert
}

\makeatletter
\def\spher@harm#1{%
	\vbox{\hbox{%
			\offinterlineskip
			\valign{&\hb@xt@2\p@{\hss$##$\hss}\vskip.2ex\cr#1\crcr}%
		}\vskip-.36ex}%
}
\def\gshone{\spher@harm{.}}
\def\gshtwo{\spher@harm{.&.}}
\def\gshthree{\spher@harm{.&.&.}}
\let\gsh\spher@harm
\makeatother

\newtheorem{corollary}{Corollary}[section]

\newtheorem{definition}{Definition}[section]
\newtheorem{assum}{Assumption}

\newtheorem{lemma}{Lemma}[section]

\newtheorem{remark}{Remark}[section]

\newtheorem{theorem}{Theorem}[section]

\newtheorem{example}{Example}

\newcolumntype{I}{!{\vrule width 1,5pt}}
\newlength\savedwidth

\newlength\savewidth

\newcounter{mnote}
\let\oldmarginpar\marginpar
\renewcommand\marginpar[1]
{\-\oldmarginpar[\raggedleft\footnotesize #1]{\raggedright\footnotesize #1}}

\makeatletter\def\@captype{table}\makeatother

\newfloatcommand{capbtabbox}{table}[][\FBwidth]
\usepackage[T1]{fontenc}
\usepackage[utf8]{inputenc}
\usepackage{authblk}
\usepackage[all]{xy}

\newcommand{\conv}[1]{{\bf conv}\left\{ {#1} \right\}}

\usepackage{graphicx}
\usepackage{tikz}
\tikzset{elegant/.style={smooth,thick,samples=50,black}}
\tikzset{eaxis/.style={->,>=stealth}}

\usepackage[misc]{ifsym}

\begin{document}
\title{
  \Large \bf An accelerated primal-dual flow for linearly constrained multiobjective optimization\thanks{This work was supported by the National Natural Science Foundation of China (Grant Nos. 12401402, 12431010, 11991024), NSFC-RGC (Hong Kong) Joint Research Program  (Grant No. 12261160365), the Science and Technology Research Program of Chongqing Municipal Education Commission (Grant Nos. KJZD-K202300505, KJQN202401624), the Natural Science Foundation of Chongqing (Grant No. CSTB2024NSCQ-MSX0329) and the Foundation of Chongqing Normal University (Grant No. 22xwB020).}}

\author[,1,3]{Hao Luo\thanks{Email: luohao@cqnu.edu.cn; luohao@cqbdri.pku.edu.cn}}
\affil[1]{National Center for Applied Mathematics in Chongqing, Chongqing Normal University, Chongqing, 401331, China}
\affil[2]{School of Mathematical Sciences, Chongqing Normal University, Chongqing, 401331, China} 
\affil[3]{Chongqing Research Institute of Big Data, Peking University,  Chongqing, 401121, China}

\author[,2]{Qiaoyuan Shu\thanks{Email: shuqy@cque.edu.cn}}

\author[,1]{Xinmin Yang\thanks{Email: xmyang@cqnu.edu.cn}}

\date{\today}
\maketitle

\begin{abstract}
 In this paper, we propose a continuous-time primal-dual approach for linearly constrained multiobjective optimization problems. A novel dynamical model, called accelerated multiobjective primal-dual flow, is presented with a second-order equation for the primal variable and a first-order equation for the dual variable. It can be viewed as an extension of the accelerated primal-dual flow by Luo [arXiv:2109.12604, 2021] for the single objective case. To facilitate the convergence rate analysis, we introduce a new merit function, which motivates the use of the feasibility violation and the objective gap to measure the weakly Pareto optimality.  By using a proper Lyapunov function, we establish the exponential decay rate in the continuous level.
After that, we consider an implicit-explicit scheme, which yields an accelerated multiobjective primal-dual method with a quadratic subproblem, and prove the sublinear rates of the feasibility violation and the objective gap, under the convex case and the strongly convex case, respectively. Numerical results are provided to demonstrate the performance of the proposed method.
\end{abstract}


\section{Introduction}
\label{sec:intro}
Multiobjective optimization problems arise from many practical fields such as engineering \cite{rangaiah2013multi}, economics \cite{fliege2014robust} and machine learning \cite{liu2024stochastic}, which aim to
to identify the so-called Pareto efficiency with multiple conflicting objectives. In some applications, including the portfolio optimization \cite{armananzas2005multiobjective}, the urban bus scheduling 
\cite{ma2021shift}, and the energy saving optimization
\cite{cui2017multi},
there are also constraints, which significantly increase the difficulty of solving such problems. Among these, a special class is the linearly constrained multiobjective optimization problem (LCMOP), which involves linear constraints and reads as follows
\begin{equation}\label{eq:lcmop}
	\tag{LCMOP}
	\min_{x\in\Bbb{R}^n}\,F(x) = \left(		f_1(x),\cdots,f_m(x)\right)^\top\quad \st Ax = b,
\end{equation}
where $A\in\Bbb{R}^{r\times n}$ and $b\in\Bbb{R}^r$ are given and each $f_j:\Bbb{R}^n\to\Bbb{R}\cup\{+\infty\}$ is proper, closed and convex. Throughout, we assume that $\Omega:=\left\{x\in\Bbb{R}^n:\,Ax = b\right\}$  is nonempty. 

When $m=1$, \cref{eq:lcmop} reduces to the standard linearly constrained optimization problem, for which Luo \cite{luo2021accelerated} proposed an accelerated primal-dual (APD) flow 
\begin{equation}\label{eq:apd}
	\tag{APD}
	\left\{
	\begin{aligned}
		{}&		\gamma x'' + (\mu+\gamma)x' +A^\top\xi+ \nabla f(x)=0,\\
		{}&\theta\xi' = A(x+x') - b,
	\end{aligned}
	\right.
\end{equation}
where $\theta'=-\theta$ and $\gamma'=\mu-\gamma$ are tailored scaling parameters and $\mu\geq0$ is the convexity parameter of the objective $f$. This work \cite{luo2021accelerated} not only establishes the exponential decay rate for the continuous level but also develops several accelerated primal-dual methods based on proper implicit, semi-implicit and explicit numerical schemes.
Motivated by this, we introduce the following accelerated multiobjective primal-dual (AMPD) flow for solving \cref{eq:lcmop}:
\begin{equation}\label{eq:intro-ampd}
	\left\{
	\begin{aligned}
		{}&		\gamma x'' + (\mu+\gamma)x'+A^\top\xi +\proj_{C(x)}(-\gamma x''-A^\top\xi)= 0,\\
		{}&\theta\xi' = A(x+x') - b,
	\end{aligned}
	\right.
\end{equation}
where $\proj_{C(x)}(\cdot)$ denotes the orthogonal projection operator onto the convex hull $C(x):=\conv{\nabla f_1(x),\cdots,\nabla f_m(x)}$. In this paper, we aim to establish the convergence rates of the feasibility violation and the objective gap for both the continuous flow \cref{eq:intro-ampd} and its proper discretization, which leads to an accelerated multiobjective primal-dual method. To our best knowledge, this constitutes the first continuous-time primal-dual framework for \cref{eq:lcmop}. 
\subsection{Single objective optimization problems}
The dynamical approach provides an alternate perspective for solving unconstrained single objective optimization problems and has tight connections with first-order methods.
This can be dated back to Polyak \cite{polyak_methods_1964}, who investigated the well-known heavy ball model that connects to the heavy ball method. About fifty years later, Su et al. \cite{su_dierential_2016} discovered the continuous analogue to Nesterov's accelerated gradient method \cite{Nesterov1983} and provided a tailored Lyapunov analysis. In recent years, this topic has attracted more attentions and we refer to \cite{Attouch2018,chen2019first, chen_luo_unified_2021,Chen2025,li2024linear,luo_differential_2021,Shi2022,Siegel2019,wibisono_variational_2016} and references therein.

Notably, such a continuous-time approach has also been extended to the  primal-dual setting. Zeng et al. \cite{zeng2022dynamical} proposed a continuous-time primal-dual dynamical system which generalizes the continuous model of Nesterov acceleration to linearly constrained optimization problems. Right after, He et al. \cite{He2021} and Attouch et al. \cite{attouch2022fast} further extended the model in \cite{zeng2022dynamical} to separable problems. Following that, fast primal-dual first-order methods based on different continuous models were proposed by Bot et al. \cite{boct2023fast}, Chen and Wei \cite{Chen2023,chen_transformed_2023}, He et al. \cite{he2022fast} and Luo \cite{luo2021accelerated,luo_universal_2024,luo_unified_2025}, and more related works \cite{chen2025-trans,he_inertial_2022,he2022second,Hu2023,luo2022primal,luo2024accelerated,zhao2023accelerated}.
\subsection{Multiobjective optimization problems}
As for the multiobjective case, apart from the extensions of classical optimization methods (first-order or second order) \cite{Ansary2015,Assuncao2021,Chen2023e,chen2024scaled,Chen2023g,ElMoudden2020,fliege_steepest_2000,tanabe2019,Tanabe2023a}, there are also several dynamical models for solving unconstrained multiobjective optimization problems. 

In \cite{Attouch2014a}, Attouch and Goudou proposed a gradient like dynamic system
\[
x'+\proj_{C(x)}(0)=0,
\]
which can be regarded as the continuous-time counterpart of the multiobjective steepest descent method \cite{fliege_steepest_2000}. Later, Attouch and Garrigos \cite{attouch2015multiibjective} introduced a second order dynamic model called inertial multiobjective gradient system
\[
x''+\gamma x'+\proj_{C(x)}(0)=0,\quad \gamma>0.
\]
Very recently, Sonntag and Peitz \cite{Sonntag2024a} proposed a multiobjective inertial gradient-like dynamical system with asymptotic vanishing damping 
\begin{equation}\label{eq:mavd}	
	x''+\frac{\alpha}{t}x' +\proj_{C(x)}(-x'')=0,\quad \alpha>0,
	\tag{MAVD}
\end{equation}
and established the convergence rate $\mathcal{O}(1/t^{2})$ for a merit function. In addition, they  \cite{Sonntag2024} considered a discrete version incorporating Nesterov acceleration  and achieving a fast rate $\mathcal{O}(1/k^{2})$. Later, Bo\c{t} and Sonntag \cite{boct2026inertial} considered an extension of \cref{eq:mavd} called multiobjective Tikhonov regularized inertial gradient system (MTRIGS)
\begin{equation}\label{key}
	\tag{MTRIGS}
	x''+\frac{\alpha}{t^q}x'+\proj_{C(x)+\frac{\beta}{t^p}x}(-x'')=0,
\end{equation}
where $\alpha,\beta>0$ and $0<q\leq 1,\,0<p\leq 2$. From a continuous perspective, Luo et al. \cite{luo_accelerated_2025} derived the continuous-time limit of the multiobjective accelerated proximal gradient method proposed by Tanabe et al. \cite{Tanabe2023a}. Building on this, Luo et al. \cite{luo_accelerated_2025} further introduced
a novel accelerated multiobjective gradient (AMG) flow  with adaptive time scaling
\begin{equation}\label{eq:amg}
	\tag{AMG}
	\gamma x''+(\mu+\gamma)x'+\proj_{C(x)}(-\gamma x'')=0.
\end{equation}
They also developed an accelerated multiobjective gradient method with an adaptive residual restart strategy and established  the sublinear rate $\mathcal{O}(L/k^{2})$ and the linear rate $\mathcal{O}\big((1-\sqrt{\mu/L})^k\big)$  for convex and strongly convex problems, respectively.

For \cref{eq:lcmop}, however, it is rare to see efficient numerical methods from the literature. In \cite{el2018multiple,el2018new}, El Moudden and El Ghali developed the multiple reduced gradient algorithm, which is based on eliminating the basic variables from the linear constraint with the full row rank assumption on $A$. Later, Cocchi and Lapucci \cite{cocchi2020} extended the augmented Lagrangian method to the multiobjective setting with general nonlinear constraints. In this work, inspired by \cite{luo2021accelerated,luo_accelerated_2025}, we are interested in developing a continuous-time primal-dual approach for \cref{eq:lcmop}. Our main contributions are summarized as follows.
\begin{itemize}
	\vskip0.1cm
	\item {\bf New merit function} Firstly, based on the standard Lagrangian gap, we introduce a merit function (cf.\cref{eq:merit}) for \cref{eq:lcmop}, which is nonnegative and vanishes at weakly Pareto optimal points. Especially, we find that this merit function can be characterized by the feasibility violation and the objective gap. This motivates the concept of an  approximate solution to the weakly Pareto optimality.
	\vskip0.1cm
	\item{\bf Novel dynamical model} Secondly, we propose an accelerated multiobjective primal-dual flow. By using the Lyapunov analysis, we show that both the feasibility violation and the objective gap decrease with an exponential rate. With proper time rescaling, our AMPD flow results in a family of dynamical models including the continuous-time primal-dual accelerated model \cite{zeng2022dynamical} for linearly constrained single objective optimization and the \cref{eq:mavd} model  \cite{Sonntag2024a} for unconstrained multiobjective optimization.
	\vskip0.1cm
	\item{\bf Multiobjective primal-dual method} Thirdly, we consider an implicit-explicit discretization scheme for our AMPD flow. This leads to an accelerated multiobjective primal-dual method with a quadratic subproblem that arises from many multiobjective gradient methods \cite{fliege_steepest_2000,tanabe2019,Tanabe2023a}. We show that,  both the feasibility violation and the objective gap admit the convergence rates $\mathcal{O}(1/k)$ and $\mathcal{O}(1/k^{2})$ for convex and strongly convex cases, respectively.
\end{itemize}
\subsection{Organization}
The rest of this paper is organized as follows. In \cref{sec:pre}, we present some preliminary results that will be used throughout the paper. In \cref{sec:ampd,sec:rate-ampd}, we introduce an accelerated multiobjective primal-dual flow and establish the exponential decay rate via the Lyapunov analysis approach. In \cref{sec:semi-ampd}, we consider an implicit-explicit discretization scheme and establish the convergence rates of the feasibility violation and objective gap. In \cref{sec:num}, we present several numerical experiments  to demonstrate the performance of the proposed method. Finally, in \cref{sec:conclu}, we provide a conclusion of our work.

\section{Preliminary}
\label{sec:pre}
\subsection{Notation}
Let $\R^d$ be the $d$-dimensional Euclidean space with the usual inner product $\dual{\cdot,\cdot}$ and the induced norm  $\nm{\cdot}:=\sqrt{\dual{\cdot,\cdot}}$. For any nonempty subset $K\subset\R^d$, define $\diam{K}:=\sup\{\nm{x}:\,x\in K\}$. Given $A\in\R^{m\times n}$, denote by $\rg A$ the column space of $A$ and  $\sigma_{\min}^+(A)$ the smallest nonzero singular value of $A$. The $d$-dimensional unit simplex is
$
\Delta_d:=\left\{\lambda\in \R^d_{+}:\, \lambda_1+\cdots+\lambda_d=1\right\}
$, where $\R^d_{+}$ is the nonnegative orthant of $\R^d$. For a collection of vectors $\{p_j\}_{j=1}^m\subset\R^d$, its convex hull is $\conv{p_j}_{j=1}^m:=\{p=\lambda_1p_1+\cdots +\lambda_mp_m\in\R^d:\,\lambda\in\Delta_m\}$.

The collection of all continuous differentiable functions from $[0,\infty)$ to $\R^d$ is denoted by the set $C^1([0,\infty);\R^d)$, and
$AC([0,\infty);\R^d)$ consists of all absolutely continuous functions from $[0,\infty)$ to $\R^d$.
Let $\mathcal{F}_L^1(\R^d)$ be the set of all $C^1$ convex functions on $\mathbb{R}^d$ with $L$-Lipschitz continuous gradients. For any $f\in\mathcal{F}_L^1(\R^d)$ we have $\nm{\nabla f(x)-\nabla f(y)}\leq L\nm{x-y}$ and
\begin{equation}
	\label{eq:L-ineq}
	0\leq    f(y)-	f(x)-\dual{\nabla f(x),y-x}\leq \frac{L}{2}\nm{x-y}^2\quad\forall\,x,y\in\R^d.
\end{equation}
All $C^1$ functions on $\R^d$ that are $\mu$-strongly convex with some $\mu \geq 0$ constitute another important function class $\mathcal S_\mu^1(\R^d)$:
\begin{equation}\label{eq:mu-ineq}
	\frac{\mu}{2}\nm{x-y}^2\leq    f(y)-	f(x)-\dual{\nabla f(x),y-x}
	\quad\forall\,x,y\in\R^d.
\end{equation}
For later use, define
$
\mathcal S^{1,1}_{\mu,L} (\R^d):= \mathcal S_\mu^1(\R^d)\cap \mathcal{F}_L^1(\R^d)$. The result given below is trivial (cf. \cite[Lemma 2.1]{luo_accelerated_2025}) but provides a useful property for the subsequent convergence rate analysis.
\begin{lemma}\label{lem:gd-lem}
	If $f\in\mathcal S_{\mu,L}^{1,1}(\R^d)$, then
	\[
	-	\dual{\nabla f(y),x-z}\leq f(z)-f(x)-\frac{\mu}{2}\nm{y-z}^2+\frac{L}{2}\nm{y-x}^2\quad\forall\,x,y,z\in\R^d.
	\]
\end{lemma}
Throughout, we impose the following assumption on the objectives of \cref{eq:lcmop}. Note that in our setting, the minimal strong convexity constant is not necessarily positive. In other words, we focus on not only the convex case $\mu=0$ but also the strongly convex case $\mu>0$, in a unified way.
\begin{assum}\label{assum:Lj-muj}
	Assume $f_j\in\mathcal S_{\mu_j,L_j}^{1,1}(\R^n)$ for $1\leq j\leq m$ with $0\leq \mu_j\leq L_j<+\infty$. For simplicity, we also denote $\mu:=\min_{1\leq j\leq m} \mu_j$ and $L:=\max_{1\leq j\leq m} L_j$.
\end{assum}
\subsection{Pareto optimality}
Given any $p,\,q\in\R^m$, we say $p$ is less than $q$ or equivalently $p<q$, if $p_j<q_j$ for all $1\leq j\leq m$. Likewise, the relation $p\leq q$ can be defined as well. A vector $y\in\Omega$ is called {\it dominated} by $x\in\Omega$ with respect to \cref{eq:lcmop} if
$F(x)\leq F(y)$ and $ F(x)\neq F(y)$.
Alternatively, when $y$ is dominated by $x$, we say $x$ dominates $y$.

A point $ x^* \in\Omega$ is called {\it  weakly Pareto optimal} or a {\it weakly Pareto optimal solution (point)} to \cref{eq:lcmop} if there does not exist
$x\in\Omega$ such that $F(x)<F( x^* )$. The weak Pareto set, denoted by $\mathcal{P}_w$, consists of all weakly Pareto optimal solutions, and the image $F(\mathcal P_w)$ of the weak Pareto set is called the {\it weak Pareto front}.

A point $ x^*\in\Omega$ is called {\it Pareto optimal} or a {\it Pareto optimal solution (point)} to \cref{eq:lcmop} if there does not exist $x\in\Omega$ that dominates $ x^* $. Denote by $\mathcal{P}$ the set of all Pareto optimal solutions, and its image $F(\mathcal P)$ is called the {\it Pareto front}. Clearly, we have $\mathcal P\subset \mathcal P_w$ by definition but the converse is not true in general.
\subsection{Optimality condition}
When each $f_j$ is smooth, the necessary optimality condition of \cref{eq:lcmop} is (cf.\cite[Theorem 3.21]{ehrgott2005multicriteria})
\begin{equation*}
	\label{eq:1st-opt-cond-conv}
	\begin{aligned}
		0={}A x^*-b,\quad
		0\in{}	A^\top\xi^*+\conv{\nabla f_j(x^*)}_{j=1}^m.
	\end{aligned}
\end{equation*}
This is equivalent to the Karush-Kuhn-Tucker (KKT) condition
\begin{equation}
	\label{eq:kkt}
	\begin{aligned}
		0={}A x^* -b,\quad
		0={}A^\top\xi^*+\proj_{C( x^*)}(-A^\top\xi^*),
	\end{aligned}
\end{equation}
where $C(x): = \conv{\nabla f_j(x)}_{j=1}^m$. If $(x^* ,\xi^*)$ satisfies the KKT condition \cref{eq:kkt}, then we call $ x^* $ a {\it Pareto critical} point of \cref{eq:lcmop}. All Pareto critical points constitute the {\it Pareto critical set}:
\[
\mathcal P_c := \{ x^*\in \Omega:\,  x^*\text{ is a Pareto critical point of \cref{eq:lcmop} }\}.
\]

Analogously to the single objective case ($m=1$), Pareto criticality is a necessary condition for the weak Pareto optimality. In the convex setting, the KKT condition \cref{eq:kkt} is also sufficient for the weak Pareto optimality. That is, for smooth and convex objectives, we have $\mathcal P_c=\mathcal P_w$; see  \cite[Corollary 3.23]{ehrgott2005multicriteria}.

\subsection{New merit function}
For $1\leq j\leq m$, denote by $Q_j(x,\xi): = f_j(x)+\dual{\xi,Ax-b}$ the usual Lagrangian function. For the single objective case $(m=1)$, the Lagrangian gap is usually used to measure the optimality of a pair $(x,\xi)$. For the multiobjective case \cref{eq:lcmop}, however, we shall consider a merit function that is nonnegative and attains zero only at weakly Pareto optimal solutions.

In this work, we introduce the following merit function
\begin{equation}\label{eq:merit}
	\Pi(x,\xi) := \sup_{z\in\Omega,\,\zeta\in\R^r} \min_{1\leq j\leq m} \pi_j(x,\xi;z,\zeta)\quad\forall\,(x,\xi)\in\R^n\times\R^r,
\end{equation}
where $
\pi_j(x,\xi;z,\zeta): =Q_j(x,\zeta)-Q_j(z,\xi)$.
Clearly, when $x\notin\Omega$, we have for any $z\in\Omega$,
\[
\sup_{\zeta\in\R^r}  \min_{1\leq j\leq m} \pi_j(x,\xi;z,\zeta) = \min_{1\leq j\leq m} [f_j(x)-f_j(z)]+ \sup_{\zeta\in\R^r}   \dual{\zeta,Ax-b}= +\infty,
\]
which implies $\Pi(x,\xi)=+\infty$ for all $(x,\xi)\in\R^n\backslash\Omega\times\R^r$.
On the other hand, for any $x\in\Omega$, we find $\min_{1\leq j\leq m} \pi_j(x,\xi;z,\zeta) =\min_{1\leq j\leq m}  [f_j(x) -f_j(z) ]$ and taking $z=x$ gives $\Pi(x,\xi)  \geq 0$. Hence, we conclude that $\Pi$ is nonnegative. Moreover, we can show that $\Pi$ attains zero at the weak Pareto points.
\begin{lemma}\label{thm:opt-Pi}
	The merit function $\Pi:\R^n\times\R^r\to\R\cup\{+\infty\}$ defined by \cref{eq:merit} is nonnegative and lower semicontinuous. Moreover, $ x^* \in\mathcal P_w$ if and only if there exists $\xi^*\in\R^r$ such that $\Pi(x^* ,\xi^*) = 0$.
\end{lemma}
\begin{proof}
	The nonnegativity has been verified by the above discussions, and by \cite[Lemmas 1.26 and 1.29]{Bauschke2011}, $\Pi$ is lower semicontinuous since $f_j$ is continuous differentiable and as well as closed (i.e. lower semicontinuous) for all $1\leq j\leq m$.
	
	It is easy to check that for $x^*\in\mathcal P_w$ and any $\xi^*\in\R^r$, we have $\Pi(x^*,\xi^*) = 0$. Let us focus on the reverse side. Assume there is a pair $ (x^*,\xi^*)\in\R^n\times \R^r$ such that $\Pi(x^* ,\xi^*)=0$. Then for all $z\in\Omega$ and $\zeta\in\R^r$, it is clear that
	\begin{equation}\label{eq:z-zeta}
		\min_{1\leq j\leq m}[f_j(x^*)-f_j(z)]+\dual{\zeta,Ax^*-b}=	 		\min_{1\leq j\leq m}\pi_j(x^*,\xi^*;z,\zeta) \leq \Pi(x^* ,\xi^*)= 0.
	\end{equation}
	Letting $\zeta=0$ in \cref{eq:z-zeta} gives  $\min_{1\leq j\leq m}[f_j( x^*)-f_j(z)]\leq  0$ for all $z\in\Omega$.
	Then by fixing $z=z_0\in\Omega$, it follows from \cref{eq:z-zeta} that
	\[
	\dual{\zeta,Ax^*-b}\leq -    \min_{1\leq j\leq m}[f_j(x^*)-f_j(z_0)]<+\infty\quad \forall \,\zeta\in \R^r.
	\]
	Since $\zeta\in \R^r$ is arbitrary, we get $Ax^*-b=0$, which implies $x^*\in\Omega$. Thus, $ x^*\in\Omega$ is a weakly Pareto point to \cref{eq:lcmop}. This completes the proof of this lemma.
\end{proof}

For the single objective case ($m=1$), we also use the feasibility violation $\nm{Ax-b}$ and the objective gap $\snm{f(x)-f^*}$ to measure the optimality of a point $x$. In the multiobjective setting, based on \cref{thm:opt-Pi}, we follow the similar idea and provide an alternate characterization of the weakly Pareto optimality.
\begin{lemma}\label{lem:Ux}
	A point $x^*$ is weakly Pareto optimal if and only if $x^*\in\Omega$ and $	U(x^*)=0$, where the objective gap function $U:\R^n\to\R\cup\{+\infty\}$ is defined by
	\begin{equation}\label{eq:obj-gap}
		U(x): = \sup_{z\in\Omega}\min_{1\leq j\leq m}  [f_j(x) -f_j(z)].
	\end{equation}
\end{lemma}
\begin{proof}
	Observe that $	\Pi(x,\xi) =+\infty$ for $x\notin\Omega$ and $	\Pi(x,\xi)=U(x)$ for $x\in\Omega$.
	Therefore, by \cref{thm:opt-Pi}, this implies immediately that $x^*\in\mathcal P_w\Longleftrightarrow \Pi(x^* ,\xi^*) = 0$ with some $\xi^*\in\R^r\Longleftrightarrow x^*\in\Omega$ and $	U(x^*)=0$.
\end{proof}

According to \cref{lem:Ux}, for any feasible point $x\in\Omega$, it is sufficient to focus on the objective gap function $\snm{U(x)}$. Motivated by this, we introduce the concept of a weakly Pareto $\epsilon$-approximation solution to \cref{eq:lcmop}.
\begin{definition}\label{def:approx}
	Let $\epsilon>0$ be given.	We call $x^\#\in\R^n$ a weakly Pareto $\epsilon$-approximation solution to \cref{eq:lcmop} if
	\[
	\nm{Ax^\#-b}\leq M_1\epsilon\quad \text{and}\quad \snm{U(x^\#)}\leq M_2\epsilon,
	\]
	where $M_1$ and $M_2$ are two generic positive constants independent on $\epsilon$ and $x^\#$.
\end{definition}

For later use, we shall restrict the objective gap function \cref{eq:obj-gap} to bounded level sets. This can be done with the following two assumptions.
\begin{assum}\label{assum:j0}
	For each $1\leq j\leq m$  the level set $\mathcal L_{f_{j}}(\alpha)=\{x\in\R^n:\,f_{j}(x)\leq \alpha\}$ is bounded for all $\alpha\in\R$.
	In other words, the quantity $	R(\alpha):=\max_{1 \leq j \leq m}R_{j}(\alpha)$ is finite, where $R_{j}(\alpha): = \sup\{\nm{x}:\,x\in \mathcal L_{f_{j}}(\alpha)\}<+\infty$.
\end{assum}
\begin{assum}\label{assum:alpha-pw}
	There exists ${\bm{\alpha}}_*\in\R^n$ such that $\mathcal L_F(\bm{\alpha}_*)\cap\Omega\neq\emptyset$. In addition, let  $\bm{\alpha}\in\R^n$ be such that $\mathcal L_F(\bm{\alpha})\cap\Omega\neq\emptyset$, then for every $x\in \mathcal L_F(\bm{\alpha})\cap\Omega$, there exists at least one  $x^*\in P_w\cap \mathcal L_F(F(x))$.
\end{assum}
\begin{lemma}
	\label{lem:u0-}
	Let $\bm{\alpha}\in\R^n$ be such that  $\mathcal L_F(\bm{\alpha})\cap\Omega\neq\emptyset$. Then we have the following.
	\begin{itemize}
		\item[(i)] Under \cref{assum:j0}, we have
		\begin{equation}\label{eq:D-alpha}
			D(\bm{\alpha}):=\sup_{F^*\in F(P_w\cap \mathcal L_F(\bm{\alpha}))}\inf_{z\in F^{-1}(F^*)\cap\Omega}\nm{z}<+\infty.
		\end{equation}
		\item[(ii)] Under \cref{assum:alpha-pw}, for any  $x\in\mathcal L_F(\bm{\alpha})$, we have
		\begin{equation}\label{eq:Pi-alter}
			U(x) =
			\sup_{F^*\in F(P_w\cap \mathcal L_F(\bm{\alpha}))}\inf_{z\in F^{-1}(F^*)\cap\Omega}\min_{1\leq j\leq m} [f_j(x) -f_j(z)].
		\end{equation}
	\end{itemize}
\end{lemma}
\begin{proof}
	Notice that $\mathcal{L}_{F}(\bm{\alpha})=\cap_{j=1}^{m}\mathcal{L}_{f_j}(\bm{\alpha}_j)$. According to \cref{assum:j0}, the level set $\mathcal{L}_{f_j}(\bm{\alpha}_j)$ is bounded for all $1\leq j\leq m$ with $R_j(\bm{\alpha}_j)<+\infty$. Then for all $z\in P_w\cap\mathcal{L}_{F}(\bm{\alpha})$, we have $\|z\|\leq\min_{1 \leq j \leq m}R_j(\bm{\alpha}_j)<+\infty$. This verifies the first claim (i).
	
	To prove the second one (ii), let us start  from the right side of \cref{eq:Pi-alter}:
	\begin{equation}\label{eq_2}
		\begin{aligned}
			{}&\sup_{F^*\in F(P_w\cap \mathcal L_F(\bm{\alpha}))}\inf_{z\in F^{-1}(F^*)\cap\Omega}\min_{1\leq j\leq m}[ f_j(x) -f_j(z) ]\\
			={}& \sup_{F^*\in F(P_w\cap \mathcal L_F(\bm{\alpha}))}\min_{1\leq j\leq m} [f_j(x)-f_j^*]
			= \sup_{z\in P_w\cap \mathcal L_F(\bm{\alpha})}\min_{1\leq j\leq m} [f_j(x)-f_j(z)].
		\end{aligned}
	\end{equation}
	By \cref{assum:alpha-pw}, for all $z\in \mathcal{L}_F(\bm{\alpha})\cap \Omega$, there exists $z^{*}\in P_w$ such that $F(z^*)\leq F(z)$. This implies the following identity
	\[
	\sup_{z\in P_w\cap \mathcal L_F(\bm{\alpha})}\min_{1\leq j\leq m} [f_j(x)-f_j(z)]
	= \sup_{z\in \mathcal L_F(\bm{\alpha})\cap \Omega}\min_{1\leq j\leq m}[f_j(x)-f_j(z)].
	\]
	In addition, for any $x\in \mathcal{L}_F(\bm{\alpha})$, it is evident that
	\begin{equation}\label{eq_3}
		\sup_{z\in \mathcal L_F(\bm{\alpha})\cap \Omega}\min_{1\leq j\leq m} [f_j(x)-f_j(z)]=\sup_{z\in \Omega}\min_{1\leq j\leq m} [f_j(x)-f_j(z)]=U(x).
	\end{equation}
	Hence, combining (\ref{eq_2}) and (\ref{eq_3}) yields \cref{eq:Pi-alter} and thus completes the proof.
\end{proof}

To the end of this section, we provide a useful lower bound of the objective gap.
\begin{lemma}\label{lem:lw-bd-U}
Let $K\subset\R^n$ be such that $\diam{K}<\infty$. Then for any $x\in K$ we have
	\[
	U(x)\geq -E_1(K,A,b)\nm{Ax-b}/\sigma_{\min}^+(A),
	\]
	where the positive constant is defined by $E_1(K,A,b):=E_2(K,A,b)\cdot \max_{1\leq j\leq m} L_j+\max_{1 \leq j \leq m} \nm{\nabla f_j(0)}$ with $E_2(K,A,b):= (1+\nm{A}/\sigma_{\min}^+(A))\diam{K} +\nm{b}/\sigma_{\min}^+(A)$.
\end{lemma}
\begin{proof}
	Let $A^+$ be the Moore--Penrose inverse of $A$, then we have $AA^+A=A$ and $AA^+b = b$ for all $b\in\rg A$. For any $x\in K$, let us consider $z_x=x-A^+(Ax-b)$. Note that we have $
	Az_x=A(x-A^+(Ax-b))=Ax-AA^+(Ax-b)=b$ and it follows from the fact $\nm{A^+}\leq 1/\sigma_{\min}^+(A)$ that
	\begin{equation}\label{eq:st-zx}
		\nm{z_x}\leq (1+\nm{A}/\sigma_{\min}^+(A))\nm{x}+\nm{b}/\sigma_{\min}^+(A)\leq E_2(K,A,b).
	\end{equation}
	In addition, we find that
	\[
	\|x-z_x\|=\|A^+(Ax-b)\|\leq\|A^+\|\|Ax-b\|\leq \|Ax-b\|/\sigma_{\min}^+(A),
	\]
	and by the triangle inequality and the $L_j$-Lipschitz continuity of $\nabla f_j$,
	\[
	\begin{aligned}
		\max_{1 \leq j \leq m} \|\nabla f_j(z_x)\| \leq {}&	
		\max_{1 \leq j \leq m} \left[\|\nabla f_j(z_x)-\nabla f_j(0)\|+\nm{\nabla f_j(0)}\right]\\
		\leq {}&\nm{z_x}\cdot \max_{1\leq j\leq m} L_j+\max_{1 \leq j \leq m}\nm{\nabla f_j(0)}\leq E_1(K,A,b).
	\end{aligned}
	\]
	Consequently, this implies that
	\[
	\begin{aligned}
		U(x)\geq {}&\min_{1\leq j\leq m}[f_j(x)-f_j(z_x)]\geq  \min_{1\leq j\leq m} \langle\nabla f_j(z_x), x-z_x\rangle\\
		\geq {}&- \max_{1\leq j\leq m} \|\nabla f_j(z_x)\|\cdot\| x-z_x\|\geq -E_1(K,A,b)\nm{Ax-b}/\sigma_{\min}^+(A).
	\end{aligned}
	\]
	This concludes the proof of this lemma.
\end{proof}

\section{Accelerated Multiobjective Primal-Dual Flow}
\label{sec:ampd}
\subsection{Continuous model}
Motivated by the \cref{eq:apd} flow \cite{luo2021accelerated} for linearly constrained optimization and the \cref{eq:amg} flow \cite{luo_accelerated_2025} for unconstrained multiobjective optimization, we propose a novel accelerated multiobjective primal-dual (AMPD) flow:
\begin{equation}\label{eq:ampd}
	\tag{AMPD}
	\left\{
	\begin{aligned}
		{}&		\gamma x'' + (\mu+\gamma)x'+A^\top\xi +\proj_{C(x)}(\beta x'-\gamma x''-A^\top\xi)= 0,\\
		{}&\theta\xi' = A(x+x') - b,
	\end{aligned}
	\right.
\end{equation}
with the initial conditions $\xi(0) = \xi_0\in\R^r,\,x(0)=x_0\in\R^n$ and $x'(0) = x_1\in\R^n$. Following \cite{luo2021accelerated}, the scaling parameters $\theta$ and $\gamma$ satisfy
\begin{equation}\label{eq:theta-gama}
	\theta'=-\theta,\quad\gamma'=\mu-\gamma,
\end{equation}
with arbitrary positive initial values: $\theta(0)=\theta_0>0$ and $\gamma(0)=\gamma_0>0$. The parameter $\beta$ in the projection term has to meet the restriction
\begin{equation}\label{eq:alpha}
	\beta+\gamma+\mu>0.
\end{equation}
The nonnegative constant $\mu\geq 0$ is the minimal convexity parameter of all objectives and has been clarified in \cref{assum:Lj-muj}. Clearly, the case $\beta=0$ is allowed, which reduces to \cref{eq:intro-ampd}.

Introduce the inertial variable $v(t): = x(t) + x'(t) $ and rewrite \cref{eq:ampd} as follows
\begin{equation}\label{eq:ampd-sys}
	\left\{
	\begin{aligned}
		{}&\theta\xi'= Av - b,\\	
		{}&		x' = v-x,\\
		{}&		\gamma v' = \mu(x-v) -A^\top\xi-	\proj_{C(x)}((\beta+\gamma )x'-\gamma v'-A^\top\xi),
	\end{aligned}
	\right.
\end{equation}
with $v(0)=v_0=x_0+x_1$. Note that, for given $u,x,w$, by \cite[Appendix A]{luo_accelerated_2025}, we have
\begin{equation}\label{eq:key-lem}
	\zeta=u-\proj_{C(x)}(w-\zeta)\quad\Longleftrightarrow\quad \zeta\in u-\argmin_{z\in C(x)}\dual{u-w,z}.
\end{equation}
This together with the condition \cref{eq:alpha} implies an equivalent differential inclusion
\begin{equation}\label{eq:ampd-pf}
	\left\{
	\begin{aligned}
		{}&		\gamma x'' + (\mu+\gamma)x'+A^\top\xi +\argmax_{z\in C(x)}\dual{x',z}\ni 0,\\
		{}&\theta\xi' = A(x+x') - b,
	\end{aligned}
	\right.
\end{equation}
which is projection-free and admits an autonomous form
\begin{equation}\label{eq:ampd-di-sys}
	\tag{AMPD-DI}
	\left\{
	\begin{aligned}
		{}&\theta\xi' = Av - b,\\	
		{}&		x' = v-x,\\
		{}&		\gamma v' \in \mu(x-v) -A^\top\xi-	\argmin_{z\in C(x)}\dual{x-v,z}.
	\end{aligned}
	\right.
\end{equation}
To see this, by \cref{eq:ampd-sys,eq:key-lem} and the fact \cref{eq:alpha}, we note that
\[
\begin{aligned}
	{}&\proj_{C(x)}((\beta+\gamma)x'-\gamma v'-A^\top\xi)
	={}	{\argmin}_{z\in C(x)}\dual{\mu(x-v)-(\beta+\gamma)x',z} \\
	={}&{\argmin}_{z\in C(x)}\big\langle({\mu+\beta+\gamma})(x-v),z\big\rangle={\argmin}_{z\in C(x)}\dual{x-v,z}.
\end{aligned}
\] 
\begin{remark}
	\label{rem:x-v-xi}
	Following the existence results for differential inclusions from \cite{aubin_differential_1984,Sonntag2024a,Sonntag2024,Tolstonogov2000}, we can establish the existence of a global solution $$(x,v,\xi)\in C^1([0,\infty);\R^n)\times AC([0,\infty);\R^n)\times C^1([0,\infty);\R^r)$$ to the autonomous system \cref{eq:ampd-di-sys}.
	Rigorously speaking, $v\in AC([0,\infty);\R^n)$ is differentiable almost everywhere in  $[0,\infty)$ but $v'$ might be discontinuous. Hence, the projection in \cref{eq:ampd-sys} shall be understood for almost all $t\geq0$. As for the original second-order model \cref{eq:ampd} or the equivalent differential inclusion \cref{eq:ampd-pf}, we shall define a solution in proper sense (cf.\cite{apidopoulos_differential_2018,luo_accelerated_2021,Sonntag2024a,Sonntag2024}). After that, it can be proved that for every global solution $(x,v,\xi)$ to  \cref{eq:ampd-di-sys}, $(x,\xi)$ is a solution to \cref{eq:ampd} and \cref{eq:ampd-pf}. We leave the detailed justifications as our future works.
\end{remark}
\subsection{Time rescaling}\label{sec:rescale}
We now give a brief discussion about the time rescaling of the \cref{eq:ampd} flow. Let $s_0\in\R$ be given. By using the rescaling rule 
\[
t(s) = \int_{s_0}^{s}\delta(r)\dd r,\quad s\geq s_0,
\]
with some continuous nonnegative function $\delta:[s_0,\infty)\to\R_+$, we can transform \cref{eq:ampd} into a rescaled one with respect to $X(s) := x(t(s)),Y(s) := \xi(t(s)),\Gamma(s) := \gamma(t(s))$ and $\Theta(s) := \theta(t(s))$. Indeed, invoking the chain rule gives
\[
\dot{X}(s) := \frac{\dd X}{\dd s}= \delta(s)x'(t(s)),\quad \ddot{X}(s) := \frac{\dd\dot{X}}{\dd s}= \delta^2(s)x''(t(s))+\frac{\dot{\delta}}{\delta}\dot{X}(s),
\]
and combining this with \cref{eq:ampd,eq:theta-gama} yields that
\begin{equation}\label{eq:Re-AMPD}
	\left\{
	\begin{aligned}
		{}&		\Gamma\delta^{-2}\ddot{X}+\delta^{-1}\big(\mu+\Gamma-\Gamma\dot{\delta}\delta^{-2}\big) \dot{X}+A^\top Y\\ {}&\qquad\qquad \qquad  =-\proj_{C(X)}\left(\big(\widetilde{\beta}\delta^{-1}+\Gamma\dot{\delta}\delta^{-3}\big)\dot{X}-\Gamma\delta^{-2}\ddot{X}-A^\top Y\right),\\
		{}&\Theta\dot{Y}= \delta\big[A(X+\delta^{-1}\dot{X}) - b\big],
	\end{aligned}
	\right.
\end{equation}
where $\widetilde{\beta}(s)=\beta(t(s))$, and the parameter equations in \cref{eq:theta-gama} turn into $\dot{\Gamma} = \delta(\mu-\Gamma)$ and  $\dot{\Theta } = -\delta\Theta$. In addition, the estimate  \cref{eq:exp-rate-E} becomes
\begin{equation} \small
	\label{eq:res-exp-rate-E}
	\begin{aligned}
		{}&	\min_{1 \leq j \leq m}\pi_j(X(s),Y(s);\widehat{x},\widehat{\xi})+\frac{\Gamma(s)}{2}\big\|X(s)+\dot{X}(s)/\delta(s)-\widehat{x}\big\|^2+\frac{\Theta(s)}{2}\big\|Y(s)-\widehat{\xi}\big\|^2\\
		\leq {}&e^{-\int_{s_0}^{s}\delta(r)\dd r}	C(\widehat{x},\widehat{\xi}),\quad s\geq s_0.
	\end{aligned}
\end{equation}

For the convex case $\mu=0$, letting $\delta=\sqrt{\Gamma}$ implies that $\delta(s) = 2\sqrt{\gamma_0}/(\sqrt{\gamma_0}(s-s_0)+2)$. With this, if $\gamma_0=\theta_0=4$ and $s_0=1$, then $\delta(s) = 2/s$ and $t(s) = 2\ln s$ for all $s\geq 1$. Thus, if we take $\beta(t)=2e^{-t}$, which satisfies \cref{eq:alpha}, then $\widetilde{\beta}(s) = \beta(t(s)) =  2s^{-2}$ and from \cref{eq:Re-AMPD} we conclude that
\begin{equation}\label{eq:3/s}
	\left\{
	\begin{aligned}
		{}&		\ddot{X}+\frac{3}{s} \dot{X}+A^\top Y +\proj_{C(X)}\big(-\ddot{X}-A^\top Y\big)= 0,\\
		{}&\dot{Y}= \frac{s}{2}\big[A(X+s/2\dot{X})- b\big].
	\end{aligned}
	\right.
\end{equation}
Also, the exponential decay in \cref{eq:res-exp-rate-E} reduces to $\mathcal O(e^{-\int_{s_0}^{s}\delta(r)\dd r})=\mathcal O(s^{-2})$. Note that when the linear constraint vanishes, \cref{eq:3/s} amounts to the \cref{eq:mavd} model \cite{Sonntag2024a,Sonntag2024}.

Furthermore, if we introduce a new variable $Z$ implicitly by that $Z+s/2\dot{Z}= Y$, then the rescaled model \cref{eq:3/s} is also equivalent to
\begin{equation}\label{eq:3/s-2nd}
	\left\{
	\begin{aligned}
		{}&		\ddot{X}+\frac{3}{s} \dot{X}+A^\top(Z+s/2\dot{Z}) +\proj_{C(X)}\big(-\ddot{X}-A^\top(Z+s/2\dot{Z})\big)= 0,\\
		{}&	\ddot{Z}+\frac{3}{s} \dot{Z}+b- A(X+s/2\dot{X}) =0.
	\end{aligned}
	\right.
\end{equation}
When the number of objectives is $m=1$, this  agrees with the continuous-time primal-dual accelerated model \cite{zeng2022dynamical} for linearly constrained single objective optimization.
\subsection{Lyapunov analysis}
\label{sec:lyap}
Let $(x,v,\xi)$ be a global solution to \cref{eq:ampd-di-sys} and $(\gamma,\theta)$ satisfy \cref{eq:theta-gama}.
Given any $\widehat{x}\in\Omega$ and $\widehat{\xi}\in\R^r$, define  the Lyapunov function 
\begin{equation}\label{eq:Li-mnag}	
	\begin{aligned}
		\mathcal E(t;\widehat{x},\widehat{\xi}) :={}& \min_{1 \leq j \leq m}\pi_j(x(t),\xi(t);\widehat{x},\widehat{\xi})+\frac{\gamma(t)}{2}\nm{v(t)-\widehat{x}}^2+\frac{\theta(t)}{2}\big\|\xi(t)-\widehat{\xi}\big\|^2,
	\end{aligned}
\end{equation}
for all $t>0$.
In our subsequent analysis, we need the following lemma, which tells us how to calculate the derivative of the Lyapunov function with respect to the time variable.
\begin{lemma}[\cite{Sonntag2024a}]
	\label{lem:min-fi-dt-fi}
	Let $z\in\R^n$ be given. If $x\in C^1([0,\infty);\R^n)$, then for all $t>0$, there exists $j(t)\in \{1,...,m\}$ such that $		\min_{1\leq j\leq m }[f_j(x(t))-f_j(z)] ={} f_{j(t)}(x(t))-f_{j(t)}(z)$, and  for almost all $t>0$, there exists $j(t)\in \{1,...,m\}$ such that
	\[
	\begin{aligned}
		\frac{\dd}{\dd t}\min_{1 \leq j \leq m}[f_j(x(t))-f_j(z)] ={}& \dual{\nabla f_{j(t)}(x(t)),x'(t)}.
	\end{aligned}
	\]
\end{lemma}
\begin{proof}
	See \cite[Lemma 4.12]{Sonntag2024a}.
\end{proof}

The Lyapunov contraction given below is the main result of this section. However, we shall mention that even if the right hand side of \cref{eq:exp-rate-E} is exponentially decay, it does not tell us the final estimates of the feasibility violation and the objective gap, as $\min_{1 \leq j \leq m}\pi_j(\cdot,\cdot;\widehat{x},\widehat{\xi})$ is not necessarily nonnegative for different $\widehat{x}\in\Omega$ and $\widehat{\xi}\in\R^r$. A complete convergence rate proof will be given in the next section.
\begin{theorem}
	\label{thm:exp-rate-E}
	Let $(x,v,\xi)$ be a global solution to \cref{eq:ampd-di-sys} and $(\gamma,\theta)$ satisfy \cref{eq:theta-gama}.
	For the Lyapunov function defined by \cref{eq:Li-mnag}	with $\widehat{x}\in\Omega$ and $\widehat{\xi}\in\R^r$, we have
	\begin{equation}\label{eq:dE}
		\frac{\dd}{\dd t} \mathcal E(t;\widehat{x},\widehat{\xi}) \leq - \mathcal E(t;\widehat{x},\widehat{\xi}),\quad\text{for almost all $t>0$}.
	\end{equation}
	This gives the exponential decay
	\begin{equation}\label{eq:exp-rate-E}
		\small
		\begin{aligned}
			{}&	\min_{1 \leq j \leq m}\pi_j(x(t),\xi(t);\widehat{x},\widehat{\xi})+\frac{\gamma(t)}{2}\nm{v(t)-\widehat{x}}^2+\frac{\theta(t)}{2}\big\|\xi(t)-\widehat{\xi}\big\|^2\leq{} e^{-t}C_0(\nm{\widehat{x}},\big\|\widehat{\xi}\big\|),
		\end{aligned}
	\end{equation}
	for all $t>0$, where the function $C_0(\cdot,\cdot):\R_+\times\R_+\to\R_+$ is defined by 
	\begin{equation}
		\label{eq:bd-est-E0}
		\begin{aligned}
			C_0(s,t):= {}&
			\left(		{\rm max}_{1\leq j\leq m}\nm{\nabla f_j(0)}+Ls\right)\left(\nm{x_0}+s\right)+L\big(\nm{x_0}^2+s^2\big)\\
			{}&\quad + t\nm{Ax_0-b}+\gamma_0\big(\nm{v_0}^2+s^2\big)+\theta_0\big(\nm{\xi_0}^2+t^2\big),
		\end{aligned}
	\end{equation}
	for all $s,t\geq0$.
\end{theorem}
\begin{proof}
	Let us first prove \cref{eq:dE}. Since $\widehat{x}\in\Omega$, we have $\pi_j(x,\xi;\widehat{x},\widehat{\xi})=f_j(x)-f_j(\widehat{x})+\big\langle \widehat{\xi},Ax-b\big\rangle$.	Thanks to \cref{lem:min-fi-dt-fi}, for almost all $t>0$, there exists $j(t)\in\{1,\cdots,m\}$ such that
	\[
	\frac{\dd}{\dd t}\min_{1\leq j\leq m}\pi_j(x,\xi;\widehat{x},\widehat{\xi})=\big	\langle x',\nabla_x\pi_{j(t)}(x,\xi;\widehat{x},\widehat{\xi})\big\rangle=\big\langle x',\nabla f_{j(t)}(x)+A^\top\widehat{\xi}\,\big\rangle.
	\]
	Thus we can replace all time derivatives with respect to the right hand side terms in \cref{eq:ampd-di-sys} and obtain that
	\[
	\begin{aligned}
		\frac{\dd}{\dd t} \mathcal E(t;\widehat{x},\widehat{\xi})
		= 	{}&
		\big\langle x',\nabla f_{j(t)}(x)+A^\top\widehat{\xi}\,\big\rangle
		+ \frac{\gamma'}{2}\nm{v-\widehat{x}}^2+	\frac{\theta'}{2}\big\|\xi-\widehat{\xi}\big\|^2\\
		{}&\qquad		+\dual{\gamma v',v-\widehat{x}}	+\big\langle\theta\xi',\xi-\widehat{\xi}\,\big\rangle\\
		=	{}&	\big\langle v-x,\nabla f_{j(t)}(x)+A^\top\widehat{\xi}\,\big\rangle
		+ \frac{\mu-\gamma}{2}\nm{v-\widehat{x}}^2-	\frac{\theta}{2}\big\|\xi-\widehat{\xi}\big\|^2\\
		{}&\qquad+\dual{\mu(x-v) -A^\top\xi-q,v-\widehat{x}}	+\big\langle Av-b,\xi-\widehat{\xi}\,\big\rangle,
	\end{aligned}
	\]		
	where $q := \mu(x-v)-\gamma v'-A^\top\xi\in \argmin_{z\in C(x)}\dual{x-v,z}$ satisfies
	\begin{equation}\label{eq:key-q-gradj}
		\dual{\nabla f_j(x)-q,v-x}\leq 0\quad\forall\,1\leq j\leq m.
	\end{equation}
	
	Recall the identity
	\begin{equation}\label{eq:3id}
		2	 \dual{x-v,v-\widehat{x} } =
		\nm{x-\widehat{x} }^2-\nm{v-\widehat{x}}^2-\nm{v-x}^2,
	\end{equation}
	which is trivial but very useful. It follows that
	\begin{equation}\label{eq:eq-I1}
		\begin{aligned}
			{}&	\dual{\mu(x-v),v-\widehat{x}}
			+ \frac{\mu-\gamma}{2}\nm{v-\widehat{x}}^2-	\frac{\theta}{2}\big\|\xi-\widehat{\xi}\big\|^2\\
			={}&\frac{\mu}{2}\nm{x-\widehat{x}}^2-\frac{\mu}{2}\nm{v-x}^2-\frac{\gamma}{2}\nm{v-\widehat{x}}^2
			-\frac{\theta}{2}\big\|\xi-\widehat{\xi}\big\|^2.
		\end{aligned}
	\end{equation}
	Invoking the splitting $v-x=v-\widehat{x}+\widehat{x}-x$ gives
	\begin{equation}\label{eq:I2}
		\begin{aligned}
			{}&\big\langle v-x,\nabla f_{j(t)}(x)+A^\top\widehat{\xi}\,\big\rangle- \dual{A^\top\xi+q,v-\widehat{x}} +\big\langle Av-b,\xi-\widehat{\xi}\,\big\rangle\\
			={}&\big\langle \widehat{x}-x,\nabla f_{j(t)}(x)+A^\top\widehat{\xi}\,\big\rangle+ \dual{\nabla f_{j(t)}(x)-q,v-\widehat{x}}\\{}&\quad + \big\langle A^{\top}(\widehat{\xi}-\xi),v-\widehat{x}\big\rangle+\big\langle Av-b,\xi-\widehat{\xi}\,\big\rangle,
		\end{aligned}
	\end{equation}
	where the last line vanishes since $\widehat{x}\in\Omega$. Similarly, using $v -\widehat{x}=v-x+v-\widehat{x}$ and the fact \cref{eq:key-q-gradj} yields that
	\[
	\begin{aligned}
		{}&\big\langle \widehat{x}-x,\nabla f_{j(t)}(x)+A^\top\widehat{\xi}\,\big\rangle+ \dual{\nabla f_{j(t)}(x)-q,v-\widehat{x}}	\\	
		={}&	\big\langle \widehat{x}-x,q+A^\top\widehat{\xi}\,\big\rangle +  \dual{\nabla f_{j(t)}(x)-q,v-x}\leq	\big\langle \widehat{x}-x,q+A^\top\widehat{\xi}\,\big\rangle.
	\end{aligned}
	\]
	Consequently, this together with \cref{eq:eq-I1,eq:I2} leads to
	\begin{equation}\label{eq:dE-mid}
		\begin{aligned}
			\frac{\dd}{\dd t}\mathcal E(t;\widehat{x},\widehat{\xi})
			\leq 	{}&	\frac{\mu}{2}\nm{x-\widehat{x}}^2+\big\langle \widehat{x}-x,q+A^\top\widehat{\xi}\,\big\rangle\\
			{}&\quad -\frac{\mu}{2}\nm{v-x}^2-\frac{\gamma}{2}\nm{v-\widehat{x}}^2
			-\frac{\theta}{2}\big\|\xi-\widehat{\xi}\big\|^2.
		\end{aligned}
	\end{equation}
	Since $q\in C(x)=\conv{\nabla f_1(x),\cdots,\nabla f_m(x)}$, assume that $q = \sum_{j=1}^{m}\lambda_j\nabla f_j(x)$ with some $\lambda\in\Delta_m$. Then it follows from \cref{eq:mu-ineq,assum:Lj-muj} that
	\[
	\begin{aligned}
		\frac{\mu}{2}\nm{x-\widehat{x}}^2+\dual{\widehat{x}-x,q} = 	{}&\sum_{j=1}^{m}\lambda_j\left(\frac{\mu}{2}\nm{x-\widehat{x}}^2+\dual{\widehat{x}-x,\nabla f_j(x)}  \right) \\
		\leq{}&\sum_{j=1}^{m}\lambda_j\left[f_j(\widehat{x})-f_j(x)\right] \leq -\min_{1\leq j \leq m}\left[f_j(x)-f_j(z)\right].
	\end{aligned}
	\]		
	Plugging this into \cref{eq:dE-mid} yields that
	\[
	\begin{aligned}
		\frac{\dd}{\dd t} \mathcal E(t;\widehat{x},\widehat{\xi})
		\leq 	- \mathcal E(t;\widehat{x},\widehat{\xi})  -\frac{\mu}{2}\nm{v-x}^2,
	\end{aligned}
	\]
	which implies \cref{eq:dE} immediately. 
	
	Now let us prove \cref{eq:exp-rate-E}. In view of \cref{eq:dE}, it is easy to obtain that $ \mathcal E(t;\widehat{x},\widehat{\xi})
	\leq 	e^{-t}\mathcal E(0;\widehat{x},\widehat{\xi})$.
	Recall the definition \cref{eq:Li-mnag}	 of the Lyapunov function, it remains to check the upper bound constant defined by \cref{eq:bd-est-E0}. To do this, let start from 
	\[
	\begin{aligned}
		\mathcal E(0;\widehat{x},\widehat{\xi})={}&\min_{1\leq j\leq m}	\pi_j(x_0,\xi_0;\widehat{x},\widehat{\xi})+\frac{\gamma_0}{2}\nm{v_0-\widehat{x}}^2+\frac{\theta_0}{2}\big\|\xi_0-\widehat{\xi}\big\|^2\\
		={}&\min_{1\leq j\leq m}[f_{j}(x_0)-f_{j}(\widehat{x})]	+\big\langle \widehat{\xi},Ax_0-b\big\rangle+\frac{\gamma_0}{2}\nm{v_0-\widehat{x}}^2+\frac{\theta_0}{2}\big\|\xi_0-\widehat{\xi}\big\|^2.
	\end{aligned}
	\]
	It is sufficient to find the upper bound of the first term. 
	Notice that by \cref{eq:L-ineq}, we have
	\[
	\begin{aligned}
		f_j(x_0)-f_j(\widehat{x})\leq{}&	 \dual{\nabla f_j(\widehat{x}),x_0-\widehat{x}} +\frac{L_j}{2}\nm{x_0-\widehat{x}}^2\\
		\leq {}&\nm{\nabla f_j(\widehat{x})}\nm{x_0-\widehat{x}}+L_j(\nm{x_0}^2+\nm{\widehat{x}}^2),
	\end{aligned}
	\]
	and 
	\begin{equation}\label{eq:nablafj-hat}
		\nm{\nabla f_j(\widehat{x})}\leq 		\nm{\nabla f_j(0)}+	\nm{\nabla f_j(\widehat{x})-\nabla f_j(0)}\leq 	\nm{\nabla f_j(0)}+L_j\nm{\widehat{x}}.
	\end{equation}
	We then arrive at 
	\begin{equation}\label{eq:f0-fhat}
		f_j(x_0)-f_j(\widehat{x})\leq\left(\nm{\nabla f_j(0)}+L_j\nm{\widehat{x}}\right)\left(\nm{x_0}+\nm{\widehat{x}}\right)+L_j(\nm{x_0}^2+\nm{\widehat{x}}^2).
	\end{equation}
	Finally, this gives $ \mathcal E(0;\widehat{x},\widehat{\xi})\leq C_0(\nm{\widehat{x}},\big\|\widehat{\xi}\big\|)$ and completes the proof.
\end{proof}

\section{Rate of Convergence}
\label{sec:rate-ampd}
As mentioned previously, the Lyapunov analysis in \cref{sec:lyap} provides nothing about the convergence rate. It takes further efforts from that to the feasibility violation and the objective gap, as defined by \cref{def:approx}. In this section, we shall complete the proof of the rate of convergence.

Let $(x,v,\xi)$ be a global solution to \cref{eq:ampd-di-sys} and $\bar{x}\in\Omega$ be arbitrarily  fixed. Define the function $C_1(\cdot):\R_+\to\R_+$ by that
\begin{equation}\label{eq:C1}
	C_1(s): = 	\left(		{\rm max}_{1\leq j\leq m}\nm{\nabla f_j(0)}+L\nm{x_0}\right)\left(\nm{x_0}+s\right)+		L\big(\nm{x_0}^2+s^2\big),\quad\forall\,s\geq0.
\end{equation}
For later use, introduce the following quantities:	
\begin{equation}\label{eq:alpha-barx}
	\begin{aligned}
		\alpha_0(\bar{x}):={}&C_0(\nm{\bar{x}},0)+C_1(\nm{\bar{x}})+{\rm max}_{1\leq j\leq m} \snm{f_j(x_0)},\\
		\alpha_1(\bar{x}):={}&	\left(1+\nm{A}/\sigma_{\min}^+(A)\right)R(\alpha_0(\bar x)) + \nm{b}/\sigma_{\min}^+(A),\\
		\alpha_2(\bar{x}):={}& L		\alpha_1(\bar{x})/\sigma_{\min}^+(A)+ {\rm max}_{1 \leq j \leq m} \nm{\nabla f_j(0)}/\sigma_{\min}^+(A),\\
		\alpha_3(\bar{x}):={}&\max\{	C_0(D(\bm\alpha(\bar{x})),0),\alpha_2(\bar{x})C_0(\alpha_1(\bar{x}),1+\alpha_2(\bar{x}))\},
	\end{aligned}
\end{equation}
where $\bm{\alpha}(\bar{x})=\max\{\bm{\alpha}_*,F(x_0)+C_1(R(\alpha_0(\bar{x})))\}$. Above, the constant $\bm{\alpha}_*$ has been declared in \cref{assum:alpha-pw}, and $R(\cdot),\,D(\cdot)$ and $C_0(\cdot,\cdot)$ are defined respectively in \cref{assum:j0,eq:D-alpha,eq:bd-est-E0}.
\begin{remark}
	For any fixed $\bar{x}\in\Omega$, all the quantities in \cref{eq:alpha-barx} are well defined and bounded constants. Actually, we can take the minimal norm element in the constraint set $\Omega$, namely $\bar{x}=A^+b=\argmin\{\nm{x}:\,Ax=b\}$, which exists uniquely.
\end{remark}

Our first goal is to establish the uniformly bound of the solution $x(t)$ over $[0,\infty)$.
\begin{lemma}\label{lem:bd-xt}
	It holds that $\nm{x(t)}\leq R(\alpha_0(\bar{x}))<+\infty$ for all $t>0$.
\end{lemma}
\begin{proof}
	Take $(\widehat{x},\widehat{\xi})=(\bar{x},0)\in \Omega\times\R^r$ in advance.
	Thanks to \cref{lem:min-fi-dt-fi,eq:exp-rate-E}, for all $t>0$, there exists $j(t)\in\{1,\cdots,m\}$ such that
	\begin{equation*}
		\begin{aligned}
			{}& f_{j(t)}(x(t))-f_{j(t)}(\bar{x})	
			=	{}	\min_{1\leq j\leq m}\pi_{j}(x(t),\xi(t);\bar{x},0)\leq e^{-t}C_0(\nm{\bar{x}},0)\leq C_0(\nm{\bar{x}},0).
		\end{aligned}
	\end{equation*}
	Similarly with \cref{eq:f0-fhat}, we can prove that $f_{j(t)}(\bar{x})-f_{j(t)}(x_0)\leq C_1(\nm{\bar{x}})$, which leads to $	f_{j(t)}(x(t))		\leq \alpha_0(\bar{x})$ and thus $x(t)\in \mathcal{L}_{f_{j(t)}}(\alpha_0(\bar{x}))$. By \cref{assum:j0} we conclude that $\nm{x(t)}\leq R(\alpha_0(\bar{x}))$ for all $t>0$. This completes the proof.
\end{proof}

Based on this, we are able to establish the exponential rate.
\begin{theorem}\label{thm:Ax-b-exp}
	We have $	\nm{Ax(t)-b}\leq  e^{-t}C_0(\alpha_1(\bar{x}),1+\alpha_2(\bar{x}))$ for all $t>0$.
\end{theorem}
\begin{proof}
	Recall that $(x,v,\xi)$ is a global solution to \cref{eq:ampd-di-sys}. Let us define
	\[
	\widetilde{x} := x-A^+(Ax-b),\quad \widetilde{\xi}:=\left\{
	\begin{aligned}
		{}&		0,&&\text{if }Ax=b,\\
		{}&	\left(1+\alpha_2(\bar{x})\right)	\frac{Ax-b}{\nm{Ax-b}},&&\text{if }Ax\neq b.
	\end{aligned}
	\right.
	\]
	It is clear that both $\widetilde{x}:[0,\infty)\to\Omega$ and $\widetilde{\xi}:[0,\infty)\to\R^r$ are well defined and we have  $\big\|\widetilde{\xi}(t)\big\|\leq 1+\alpha_2(\bar{x})$. By \cref{lem:bd-xt}, a similar argument with that of \cref{eq:st-zx} implies $\nm{\widetilde{x}(t)}\leq
	\alpha_1(\bar{x})$, and it is important to see that $
	{}\big\langle\widetilde{\xi}(t),Ax(t)-b\big\rangle\geq(1+\alpha_2(\bar{x})) \nm{Ax(t)-b}$ and $\nm{x(t)-\widetilde{x}(t)}\leq  \nm{Ax(t)-b}/\sigma_{\min}^+(A)$. Analogously to \cref{eq:nablafj-hat} we have for all $1\leq j\leq m$, that $	\nm{\nabla f_j(\widetilde{x}(t))}\leq 		 {\rm max}_{1 \leq j \leq m}\nm{\nabla f_j(0)}+L\alpha_1(\bar{x})$
	and $	\dual{\nabla f_j(\widetilde{x}(t)),x(t)-\widetilde{x}(t)}\geq -\alpha_2(\bar{x})\nm{Ax(t)-b}$ for all $t>0$.
	
	Now, let $\tau>0$ be arbitrary. Take $(\widehat{x},\widehat{\xi})=(\widetilde{x}(\tau),\widetilde{\xi}(\tau))\in \Omega\times\R^r$ in advance. Again, by \cref{lem:min-fi-dt-fi,eq:exp-rate-E}, for all $t>0$, there exists $j(t)\in\{1,\cdots,m\}$ such that
	\begin{equation*}\small
		\begin{aligned}
			{}& f_{j(t)}(x(t))-f_{j(t)}(\widehat{x})	+\big\langle \widehat{\xi},Ax(t)-b\big\rangle
			=	{}\min_{1\leq j\leq m}\pi_{j}(x(t),\xi(t);\widehat{x},\widehat{\xi})\leq e^{-t}C_0(\nm{\widehat{x}},\big\|\widehat{\xi}\big\|),
		\end{aligned}
	\end{equation*}
	where by \cref{eq:bd-est-E0} we get $C_0(\nm{\widehat{x}},\big\|\widehat{\xi}\big\|)=C_0(\nm{\widetilde{x}(\tau)},\big\|\widetilde{\xi}(\tau)\big\|)\leq  C_0(\alpha_1(\bar{x}),1+\alpha_2(\bar{x}))$.
	Especially, at time $t=\tau$, we have
	\[
	f_{j(\tau)}(x(\tau))-f_{j(\tau)}(\widetilde{x}(\tau))+\big\langle \widetilde{\xi}(\tau),Ax(\tau)-b\big\rangle \leq e^{-\tau}C_0(\alpha_1(\bar{x}),1+\alpha_2(\bar{x})).
	\]
	On the other hand, we find that
	\begin{equation*}
		\begin{aligned}
			{}&f_{j(\tau)}(x(\tau))-f_{j(\tau)}(\widetilde{x}(\tau))+\big\langle \widetilde{\xi}(\tau),Ax(\tau)-b\big\rangle \\
			\geq{}&
			\dual{\nabla f_{j(\tau)}(\widetilde{x}(\tau)), x(\tau)-\widetilde{x}(\tau)}
			+(1+\alpha_2(\bar{x}))\nm{Ax(\tau)-b}
			\geq
			\nm{Ax(\tau)-b}.
		\end{aligned}
	\end{equation*}
	Consequently, we obtain $\nm{Ax(\tau)-b}\leq e^{-\tau}C_0(\alpha_1(\bar{x}),1+\alpha_2(\bar{x}))$ for any $\tau>0$. This finishes the proof immediately.
\end{proof}

\begin{theorem}\label{thm:est-Uxt}
	For all $t>0$, it holds that $		\snm{U(x(t))}\leq \alpha_3(\bar{x}) e^{-t}$.
\end{theorem}
\begin{proof}
	Let us firstly prove that $x(t)\in\mathcal L_F(\bm{\alpha}(\bar{x}))$ for all $t>0$. Thanks to \cref{lem:bd-xt}, we have $\nm{x(t)}\leq R(\alpha_0(\bar{x}))$. Analogously to \cref{eq:f0-fhat}, for $1\leq j\leq m$, it is not hard to obtain $	f_{j}(x(t))-f_{j}(x_0)\leq C_1(\nm{x(t)})\leq C_1(R(\alpha_0(\bar{x})))$ as $C_1(\cdot)$ defined by \cref{eq:C1} is increasing,
	which further implies that $F(x(t))\leq
	F(x_0)+C_1(R(\alpha_0(\bar{x})))\leq\bm{\alpha}(\bar{x})$.
	
	Then, for any $\widehat{x}\in\Omega$, according to \cref{eq:exp-rate-E}, we have
	\begin{equation*}
		\begin{aligned}
			{}& 	\min_{1 \leq j \leq m}\left[	f_j(x(t))-f_j(\widehat{x})\right]
			=	{}	\min_{1\leq j\leq m}\pi_{j}(x(t),\xi(t);\widehat{x},0)\leq e^{-t}C_0(\nm{\widehat{x}},0),\quad\forall\,t>0.
		\end{aligned}
	\end{equation*}
	According to \cref{assum:alpha-pw}, $\mathcal L_F(\bm{\alpha}_*)\cap\Omega$ is nonempty and so is $\mathcal L_F(\bm{\alpha}(\bar{x}))\cap\Omega$ since $\bm{\alpha}(\bar{x})=\max\{\bm{\alpha}_*,F(x_0)+C_1(R(\alpha_0(\bar{x})))\}$.
	Noticing that $x(t)\in\ \mathcal L_F({\bm{\alpha}}(\bar{x}))$, by using \cref{lem:u0-} we obtain the upper bound estimate
	\[
	\begin{aligned}
		{}&	U(x(t))=\sup_{F^*\in F(P_w\cap \mathcal L_F(\bm{\alpha}(\bar{x})))}\inf_{\widehat{x}\in F^{-1}(F^*)\cap\Omega}\min_{1\leq j\leq m} \left[f_j(x(t)) -f_j(\widehat{x})\right]\\
		\leq {}&e^{-t}\sup_{F^*\in F(P_w\cap \mathcal L_F(\bm{\alpha}(\bar{x})))}\inf_{\widehat{x}\in F^{-1}(F^*)\cap\Omega}C_0(\nm{\widehat{x}},0)\leq e^{-t}
		C_0(D(\bm\alpha(\bar{x})),0).
	\end{aligned}
	\]
	In addition, by \cref{lem:bd-xt} we have $\nm{x(t)}\leq R(\alpha_0(\bar{x}))$ for all $t>0$, and thus by using \cref{lem:lw-bd-U,thm:Ax-b-exp}, we find the lower bound estimate
	\[
	U(x(t))\geq -\alpha_2(\bar{x})\nm{Ax(t)-b}\geq -\alpha_2(\bar{x})C_0(\alpha_1(\bar{x}),1+\alpha_2(\bar{x}))e^{-t}.
	\]
	Finally, we get the desired estimate $		\snm{U(x(t))}\leq \alpha_3(\bar{x}) e^{-t}$ and complete the proof.
\end{proof} 

\section{An Accelerated Multiobjective Primal-Dual Method}
\label{sec:semi-ampd}
Our continuous model together with its Lyapunov analysis and convergence rate proof paves the way for designing and analyzing first-order  methods for solving \cref{eq:lcmop}. In this part, we present an implicit-explicit (IMEX) scheme that results in an accelerated multiobjective primal-dual method with a quadratic programming subproblem; see \cref{algo:amd-qp}. Based on the discrete Lyapunov analysis, we establish the convergence rates $\mathcal O(1/k)$ and $\mathcal O(1/k^2)$ of the feasibility violation $\nm{Ax_k-b}$ and the objective gap $\snm{U(x_k)}$, respectively for convex case $\mu=0$ and strongly convex case $\mu>0$.
\subsection{Numerical scheme}
Observe that \cref{eq:ampd} involves the second-order derivative in time. To avoid this, let us start from the equivalent first-order system \cref{eq:ampd-sys} with $\beta=-\mu$ (which satisfies \cref{eq:alpha}):
\begin{equation}\label{eq:ampd-sys-dis}
	\left\{
	\begin{aligned}
		{}&\theta\xi'= Av - b,\\	
		{}&		x' = v-x,\\
		{}&		\gamma v' = \mu(x-v) -A^\top\xi-	\proj_{C(x)}(w-A^\top\xi),
	\end{aligned}
	\right.
\end{equation}
with $w:=-\mu(v-x)+\gamma  x'-\gamma v'$. Given the current iteration $(x_k,v_k,\xi_k)$ and the step size $\alpha_k>0$, compute the predictions $	y_k = (x_k+\alpha_kv_k)/(1+\alpha_k) $ and $\whk=
\xi_k+ \alpha_k/\theta_k \left(
Av_k-b\right)$,
and update $(x_{k+1},v_{k+1},\xi_{k+1})$ by the IMEX scheme for \cref{eq:ampd-sys-dis}:
\begin{subnumcases}{\label{eq:imex-qp}}
	\label{eq:imex-qp-xik1}
	\theta_k	\dfrac{\xi_{k+1}-\xi_k}{\alpha_k}={}Av_{k+1}-b,\\
	\label{eq:imex-qp-xk1}
	\dfrac{x_{k+1}-x_k}{\alpha_k}={}v_{k+1}-x_{k+1},\\
	\label{eq:imex-qp-vk1}
	\gamma_k\dfrac{v_{k+1}-v_k}{\alpha_k}={}\mu(y_k-v_{k+1})-A^\top\whk-\textbf{proj}_{C(y_k)}\big(w_{k+1}-A^\top\whk\big),
\end{subnumcases}
where $	w_{k+1} ={} -\mu(v_{k+1}-y_k)+\gamma_k(x_{k+1}-x_k)/\alpha_k-\gamma_k(v_{k+1}-v_k)/\alpha_k$.
The scaling parameter equations in \cref{eq:theta-gama} are discretized implicitly
\begin{equation}\label{eq:thetak-gamak}
	\frac{\theta_{k+1}-\theta_k}{\alpha_k}=-\theta_{k+1},\quad
	\frac{\gamma_{k+1}-\gamma_k}{\alpha_k}=\mu-\gamma_{k+1}.
\end{equation}

Let us discuss the solvability of the IMEX scheme \cref{eq:imex-qp}. Based on \eqref{eq:imex-qp-vk1}, a simple calculation leads to
\[
(\gamma_k+\mu\alpha_k)/\alpha_kv_{k+1} =\bar{v}_k+\gamma_kx_k/(1+\alpha_k)-\proj_{C(y_k)}\left(\bar{v}_k-\eta_kv_{k+1}\right),
\]
where $\eta_k=	(\gamma_k+\mu\alpha_k)/\alpha_k-\gamma_k/(1+\alpha_k)$ and
$
\bar{v}_k=	(\gamma_kv_k+\mu\alpha_k y_k)/\alpha_{k}-\gamma_kx_k/(1+\alpha_k)-A^\top\whk
$. This is an implicit equation and  by \cite[Appendix A]{luo_accelerated_2025}, we have \[
v_{k+1} = \alpha_k\big(\bar{v}_k+\gamma_kx_k/(1+\alpha_k)-v_{k+1}^{\rm QP}\big)/(\gamma_k+\mu\alpha_k),
\]
with $v_{k+1}^{\rm QP}={}\proj_{C(y_k)}\big({\alpha_k^{-1}\gamma_k(v_k-x_k)+\mu (y_k-x_k)- A^\top \whk}\big)$.
Note that this involves the projection onto $C(y_k)=\conv{\nabla f_1(y_k),\cdots,\nabla f_m(y_k)}$ and can be transformed into a quadratic programming over the probability simplex $\Delta_m$. This is also similar with the dual approach used in multiobjective gradient methods; see \cite{fliege_steepest_2000,luo_accelerated_2025,Sonntag2024,tanabe2019,Tanabe2023a}.

\begin{algorithm}[H]
	\caption{AMPD-QP for solving \cref{eq:lcmop} with $f_j\in\mathcal S_{\mu_j,L_j}^{1,1}(\R^n)$}
	\label{algo:amd-qp}
	\begin{algorithmic}[1]
		\REQUIRE  Problem parameters: $L=\max_{1\leq j\leq m}L_j>0$ and $\mu=\min_{1 \leq j \leq m}\mu_j\geq 0$.\\
		~~~~~~Initial values: $	x_0,v_0\in\R^n$ and $\theta,\,\gamma_0>0$.\\
		~~~~~~KKT residual tolerance: $ \epsilon>0$.	
		\FOR{$k=0,1,\cdots$}
		\STATE Compute  ${\rm KKT}(x_k,\xi_k)$ by \cref{eq:kkt-res}.
		\IF{${\rm KKT}(x_k,\xi_k)\leq \epsilon$}
		\RETURN An approximated solution $x_k\in \mathbb{R}^n$ to \cref{eq:lcmop}.
		\ELSE
		\STATE  Find $\alpha_k>0$ satisfying \cref{eq:cond-ak}.
		\STATE   $\theta_{k+1}=\theta_k/(1+\alpha_k)$,  $\gamma_{k+1} = (\gamma_k + \mu\alpha_k)/(1+\alpha_k)$
		\STATE    $y_k = (x_k + \alpha_k v_k)/(1+\alpha_k),\,\brk = \xi_k+\alpha_k\theta_k^{-1}(Av_k-b)$
		\STATE   $v_{k+1}^{\mathrm{QP}} = \proj_{C(\brxk)}\left( {\alpha_k}^{-1} \gamma_k(v_k-x_k)+ \mu(y_k-x_k) -A^\top\brk\right)$
		\STATE  $v_{k+1} = (\gamma_k+\mu\alpha_k)^{-1}\big(\gamma_k v_k + \mu\alpha_k y_k -\alpha_kA^\top\brk- \alpha_k v_{k+1}^{\mathrm{QP}}\big)$
		\STATE  $\xi_{k+1} = \xi_k + \alpha_k\theta_k^{-1} (Av_{k+1}-b),x_{k+1} = (x_k+\alpha_k v_{k+1})/(1+\alpha_k)$
		\ENDIF
		\ENDFOR
		\ENSURE An  approximated solution $x_k\in \mathbb{R}^n$ to \cref{eq:lcmop}.
	\end{algorithmic}
\end{algorithm}

In \cref{algo:amd-qp}, we present an equivalent form of the IMEX scheme \cref{eq:imex-qp} with the step size constraint
\begin{equation}\label{eq:cond-ak}
	\alpha_k^2(L\theta_{k}/(1+\alpha_k)+\nm{A}^2) \leq \gamma_k\theta_k.
\end{equation}
It is called an Accelerated Multiobjective Primal-Dual method with a Quadratic Programming subproblem (AMPD-QP for short).
For a given pair $(\theta_k,\gamma_k)$, it is easy to choose proper $\alpha_k>0$ satisfying the constraint \cref{eq:cond-ak}. For instance, we can simply take $\alpha_k^2(L\theta_k+\nm{A}^2)=\gamma_k\theta_k$.
As the $\epsilon$-approximation solution given by \cref{def:approx} is not convenient for us to check. Thus, in \cref{algo:amd-qp}, we also propose a stopping criterion via the KKT residual
\begin{equation}\label{eq:kkt-res}
	{\rm KKT}(x_k,\xi_k):=\sqrt{\nm{Ax_k-b}^2+\|A^\top\xi_k+\proj_{C( x_k)}(-A^\top\xi_k)\|^2}.
\end{equation}

\subsection{Discrete Lyapunov analysis}
For the convergence rate proof, we shall provide the discrete Lyapunov analysis.
\begin{lemma}
	\label{lem:key-id-imex-qp}
	Let $\{x_k,v_k,\xi_k\}$ be generated by \cref{eq:imex-qp}, then we have
	\begin{equation}\label{eq:key-id-imex-qp}
		\begin{aligned}
			{}&	\big\langle \proj_{C(y_{k})}\big(w_{k+1}-A^\top\whk\big),x_{k+1}-x_k\big\rangle = {}{\rm max}_{1\leq j\leq m}	\dual{\nabla f_j(y_{k}) ,  x_{k+1}-x_k}.
		\end{aligned}
	\end{equation}
\end{lemma}
\begin{proof}
	By \eqref{eq:imex-qp-vk1} we claim that
	\[
	\begin{aligned}
		{}&	\big\langle \nabla f_j(y_{k})-	\proj_{C(y_{k})}\big(w_{k+1}-A^\top\whk\big),		
		x_{k+1}-x_k
		\big\rangle\leq 0,
	\end{aligned}
	\]
	for all $1\leq j\leq m$.
	Clearly, this exists $\lambda_k=(\lambda_{k,1},\cdots,\lambda_{k,m})^\top\in\Delta_m$ such that $\proj_{C(y_{k})}\big(w_{k+1}-A^\top\whk\big)=\sum_{j=1}^{m}\lambda_{k,j}\nabla f_j(y_k)$. Consequently, this implies that
	\[
	\begin{aligned}
		{}&	\max_{1\leq j\leq m}	\dual{\nabla f_j(y_{k}),x_{k+1}-x_k}
		\leq{} \big\langle \proj_{C(y_{k})}\big(w_{k+1}-A^\top\whk\big),x_{k+1}-x_k\big\rangle\\
		={}&\sum_{j=1}^{m}\lambda_{k,j}	\dual{ \nabla f_j(y_{k}),x_{k+1}-x_k }
		\leq {}\max_{1\leq j\leq m}	\dual{\nabla f_j(y_{k}),x_{k+1}-x_k }.
	\end{aligned}
	\]
	This leads to the identity \cref{eq:key-id-imex-qp} and finishes the proof.
\end{proof}

Following \cref{eq:Li-mnag}, define a discrete Lyapunov function by that
\begin{equation}\label{eq:Ek}
	\begin{aligned}
		\mathcal E_k(\widehat{x},\widehat{\xi}) :={}& \min_{1 \leq j \leq m}\pi_j(x_k,\xi_k;\widehat{x},\widehat{\xi})+\frac{\gamma_k}{2}\nm{v_k-\widehat{x}}^2+\frac{\theta_k}{2}\big\|\xi_k-\widehat{\xi}\big\|^2,\quad k\in\mathbb N,
	\end{aligned}
\end{equation}
where $\widehat{x}\in\Omega$ and $\widehat{\xi}\in\R^r$ are arbitrary.
\begin{theorem}
	\label{thm:conv-imex-qp}
	Let $\{x_k,v_k,\xi_k\}$ be generated by \cref{eq:imex-qp} with the step size constraint \cref{eq:cond-ak}. Then for any $\widehat{x}\in\Omega$ and $\widehat{\xi}\in\R^r$, we have
	\begin{equation}\label{eq:diff-Ek}		\mathcal{E}_{k+1}(\widehat{x},\widehat{\xi})-\mathcal{E}_{k}(\widehat{x},\widehat{\xi})\leq -\alpha_k\mathcal{E}_{k+1}(\widehat{x},\widehat{\xi}),\quad k\in\mathbb N,
	\end{equation}
	which implies that
	\begin{equation}\label{eq:conv-imex-qp-2}
		\min_{1\leq j\leq m}\pi_j(x_k,\xi_k;\widehat{x},\widehat{\xi})+\frac{\gamma_k}{2}\nm{v_k-\widehat{x}}^2+\frac{\theta_k}{2}\big\|\xi_{k}-\widehat{\xi}\big\|^2 \leq \theta_k/\theta_0 C_0(\nm{\widehat{x}},\big\|\widehat{\xi}\|),
	\end{equation}
	where $C_0(\cdot,\cdot)$ is defined by \cref{eq:bd-est-E0}.
\end{theorem}
\begin{proof}
	Notice that by \cref{eq:thetak-gamak} we have $\theta_{k+1}=\theta_k/(1+\alpha_k)$. Therefore, if \cref{eq:diff-Ek}	holds true, then it follows directly that
	\[
	\mathcal{E}_{k}(\widehat{x},\widehat{\xi})\leq \frac{\mathcal{E}_{k-1}(\widehat{x},\widehat{\xi})}{1+\alpha_{k-1}}=\frac{\theta_k\mathcal{E}_{k-1}(\widehat{x},\widehat{\xi})}{\theta_{k-1}}\leq \cdots\leq
	\frac{\theta_k}{\theta_0} \mathcal{E}_0(\widehat{x},\widehat{\xi}).
	\]
	Similarly with the proof of \cref{eq:exp-rate-E}, we have $\mathcal{E}_0(\widehat{x},\widehat{\xi})\leq C_0(\nm{\widehat{x}},\big\|\widehat{\xi}\|)$, which implies \cref{eq:conv-imex-qp-2}. Henceforth, it is sufficient to \cref{eq:diff-Ek}. Observe the decomposition
	\[
	\begin{aligned}
		\mathcal{E}_{k+1}(\widehat{x},\widehat{\xi})-\mathcal{E}_{k}(\widehat{x},\widehat{\xi})={}& \min_{1\leq j\leq m}\pi_j(x_{k+1},\xi_{k+1};\widehat{x},\widehat{\xi})- \min_{1\leq j\leq m}\pi_j(x_k,\xi_k;\widehat{x},\widehat{\xi})\\
		{}&\quad +\frac{\gamma_{k+1}}{2}\left\|v_{k+1}-\widehat{x}\right\|^{2}-\frac{\gamma_{k}}{2}\left\|v_{k}-\widehat{x}\right\|^{2}\\
		{}&\qquad+ \frac{\theta_{k+1}}{2}\big\|\xi_{k+1}-\widehat{\xi}\big\|^2
		-\frac{\theta_{k}}{2}\big\|\xi_{k}-\widehat{\xi}\big\|^2
		:={}\mathbb I_1+\mathbb I_2+\mathbb I_3.
	\end{aligned}
	\]
	It can be proved that
	\begin{equation}\label{eq:est-I1+I2}\small
		\begin{aligned}
			\mathbb I_1+\mathbb I_2
			\leq &-\alpha_k	\min_{1\leq j\leq m}\pi_j(x_{k+1},\xi_{k+1};\widehat{x},\widehat{\xi})-\frac{\alpha_k\gamma_{k+1}}{2}\nm{v_{k+1}-\widehat{x}}^2\\
			{}&\quad -\frac{\gamma_{k}}{2}\left\|v_{k+1}-v_{k}\right\|^{2}+\frac{L(1+\alpha_k)}{2}\nm{x_{k+1}-y_k}^2
			-\alpha_{k}\big\langle Av_{k+1}-b,\whk-\widehat{\xi}\,\big\rangle,
		\end{aligned}
	\end{equation}
	and
	\begin{equation}\label{eq:est-I3}
		\mathbb I_3\leq  -\frac{\alpha_k\theta_{k+1}}{2}\big\|\xi_{k+1}-\widehat{\xi}\big\|^2
		+\frac{\theta_{k}}{2}\big\|\xi_{k+1}-\whk\big\|^2
		+\alpha_{k}\big\langle Av_{k+1}-b,\whk-\widehat{\xi}\,\big\rangle.
	\end{equation}
	Combining these two estimates gives
	\[
	\begin{aligned}
		\mathcal{E}_{k+1}(\widehat{x},\widehat{\xi})-\mathcal{E}_{k}(\widehat{x},\widehat{\xi})\leq {}& -\alpha_k		\mathcal{E}_{k+1}(\widehat{x},\widehat{\xi})-\frac{\gamma_{k}}{2}\left\|v_{k+1}-v_{k}\right\|^{2}\\
		{}&\quad +\frac{L(1+\alpha_k)}{2}\nm{x_{k+1}-y_k}^2	+\frac{\theta_{k}}{2}\big\|\xi_{k+1}-\whk\big\|^2.
	\end{aligned}
	\]
	Note that $\xi_{k+1}-\whk = \alpha_k/\theta_kA(v_{k+1}-v_k)$ and
	\[
	x_{k+1}-y_k = \frac{x_k+\alpha_k v_{k+1}}{1+\alpha_k}-\frac{x_k+\alpha_k v_k}{1+\alpha_k} =\frac{ \alpha_k(v_{k+1}-v_k)}{1+\alpha_k},
	\]
	which further implies that
	\[
	\begin{aligned}
		\mathcal{E}_{k+1}(\widehat{x},\widehat{\xi})-\mathcal{E}_{k}(\widehat{x},\widehat{\xi})\leq {}& -\alpha_k		\mathcal{E}_{k+1}(\widehat{x},\widehat{\xi}) +\frac{\Delta_k}{2\theta_k}\left\|v_{k+1}-v_{k}\right\|^{2},
	\end{aligned}
	\]
	where $\Delta_k:=L\alpha_k^2\theta_{k}/(1+\alpha_k)+\alpha_k^2\nm{A}^2-\gamma_{k}\theta_k\leq 0$.
	Consequently, the contraction estimate \cref{eq:diff-Ek} follows immediately.
	
	To complete the proof of this theorem, it remains to verify \cref{eq:est-I3,eq:est-I1+I2}. An evident calculation yields that
	\begin{align}
		\mathbb I_3={}&\frac{\theta_{k+1}-\theta_{k}}{2}\big\|\xi_{k+1}-\widehat{\xi}\big\|^2
		+\frac{\theta_{k}}{2}
		\left(\big\|\xi_{k+1}-\widehat{\xi}\big\|^2 -
		\big\|\xi_{k}-\widehat{\xi}\big\|^2\right)\notag\\
		={}&	-\frac{\alpha_k\theta_{k+1}}{2}\big\|\xi_{k+1}-\widehat{\xi}\big\|^2
		-\frac{\theta_{k}}{2}\nm{\xi_{k+1}-\xi_k}^2
		+\theta_{k}\big\langle\xi_{k+1}-\xi_k,\xi_{k+1}-\widehat{\xi}\big\rangle.
		\label{eq:I1-im-x-im-l}
	\end{align}
	We then insert $\whk$ into the last cross term to obtain
	\[
	\begin{aligned}
		\mathbb I_3={}&-\frac{\alpha_k\theta_{k+1}}{2}\big\|\xi_{k+1}-\widehat{\xi}\big\|^2
		-\frac{\theta_{k}}{2}\nm{\xi_{k+1}-\xi_k}^2\\
		{}&	\quad+\theta_{k}\big\langle \xi_{k+1}-\xi_k,\whk-\widehat{\xi}\,\big\rangle+\theta_{k}\big\langle \xi_{k+1}-\xi_k,\xi_{k+1}-\whk\big\rangle.
	\end{aligned}
	\]
	Applying the three-term identity \cref{eq:3id} to the last cross term gives
	\[\small
	\begin{aligned}
		\mathbb I_3
		=&-\frac{\alpha_k\theta_{k+1}}{2}\big\|\xi_{k+1}-\widehat{\xi}\big\|^2
		+\frac{\theta_{k}}{2}\big(\big\|\xi_{k+1}-\whk\big\|^2-\big\|\xi_{k}-\whk\big\|^2\big)+\theta_{k}\big\langle \xi_{k+1}-\xi_k,\whk-\widehat{\xi}\,\big\rangle
		.
	\end{aligned}
	\]
	Dropping the negative term $-\big\|\xi_{k}-\whk\big\|^2$ and rewriting the last term (cf.\eqref{eq:imex-qp-xik1}), we get the desired estimate \cref{eq:est-I3}.
	
	To the end, let us prove \cref{eq:est-I1+I2}. It follows from  \cref{lem:gd-lem} that
	\[
	\begin{aligned}
		\mathbb I_1
		\leq 	{}&\max_{1\leq j\leq m}\big[f_j(x_{k+1})-f_j(x_k)+\big\langle\widehat{\xi},A(x_{k+1}-x_k)\big\rangle\big]\\
		\leq  {}&\max_{1\leq j \leq m}\dual{\nabla f_j(y_k),x_{k+1}-x_k}+\frac{L}{2}\nm{x_{k+1}-y_k}^2+\big\langle\widehat{\xi},A(x_{k+1}-x_k)\big\rangle.
	\end{aligned}
	\]
	Here, we mention an extension of the three-term identity \cref{eq:3id}:
	$$
	a\nm{u}^2-b\nm{w}^2=(a-b)\nm{u}^2-b\nm{u-w}^2+2b\dual{u,u-w},
	$$
	which holds true for all $a,b\in\R$ and $u,w\in\R^n$. Applying this to $\mathbb I_2$ gives
	\[\small
	\begin{aligned}
		\mathbb I_2={}&\frac{\gamma_{k+1}-\gamma_k}{2}\left\|v_{k+1}-\widehat{x}\right\|^{2}-\frac{\gamma_{k}}{2}\left\|v_{k+1}-v_{k}\right\|^{2}+\gamma_{k}\dual{v_{k+1}-v_{k},v_{k+1}-\widehat{x}}\\
		={}&\frac{\mu\alpha_{k}-\alpha_{k}\gamma_{k+1}}{2}\left\|v_{k+1}-\widehat{x}\right\|^{2}-\frac{\gamma_{k}}{2}\left\|v_{k+1}-v_{k}\right\|^{2}+\gamma_{k}\dual{v_{k+1}-v_{k},v_{k+1}-\widehat{x}} \quad\left(\text{by \cref{eq:thetak-gamak}}\right)\\
		={}&\big\langle\mu\alpha_{k}(y_{k}-v_{k+1})-\alpha_k A^\top\whk-\alpha_{k}\textbf{proj}_{C(y_{k})}\big(w_{k+1}-A^\top\whk\big),v_{k+1}-\widehat{x}\big\rangle\\
		{}&\quad-\frac{\gamma_{k}}{2}\left\|v_{k+1}-v_{k}\right\|^{2}+\frac{\mu\alpha_{k}-\alpha_{k}\gamma_{k+1}}{2}\left\|v_{k+1}-\widehat{x}\right\|^{2}\quad\left(\text{by \eqref{eq:imex-qp-vk1}}\right).
	\end{aligned}
	\]
	Similarly with the identity \cref{eq:3id}, we have
	\[
	\mu\alpha_k\dual{y_k-v_{k+1},v_{k+1}-\widehat{x}}=\frac{\mu\alpha_k}{2}\big(\nm{y_{k}-\widehat{x}}^2-\nm{v_{k+1}-\widehat{x}}^2-\nm{v_{k+1}-y_{k}}^2\big).
	\]
	Thanks to \eqref{eq:imex-qp-xk1} and \cref{lem:key-id-imex-qp}, we find that
	\[\small
	\begin{aligned}
		{}&	- \alpha_{k}\big\langle\textbf{proj}_{C(y_{k})}\big(w_{k+1}-A^\top\whk\big),v_{k+1}-\widehat{x}\big\rangle\\
		={}&-\big\langle\textbf{proj}_{C(y_{k})}\big(w_{k+1}-A^\top\whk\big),x_{k+1}-x_k\big\rangle-\alpha_{k}\big\langle\textbf{proj}_{C(y_{k})}\big(w_{k+1}-A^\top\whk\big),x_{k+1}-\widehat{x}\big\rangle\\
		={}&			
		-					{\rm max}_{1\leq j \leq m}
		\dual{\nabla f_j(y_k),x_{k+1}-x_k}-\alpha_{k}\big\langle\textbf{proj}_{C(y_{k})}\big(w_{k+1}-A^\top\whk\big),x_{k+1}-\widehat{x}\big\rangle.
	\end{aligned}
	\]
	Let  $\lambda_k=(\lambda_{k,1},\cdots,\lambda_{k,m})^\top\in\Delta_m$ be such that the projection admits the presentation  $\proj_{C(y_{k})}\big(w_{k+1}-A^\top\whk\big)=\sum_{j=1}^{m}\lambda_{k,j}\nabla f_j(y_k)$. Then we obtain from \cref{lem:gd-lem,assum:Lj-muj} that
	\[
	\begin{aligned}
		{}& -\alpha_{k}\big\langle\textbf{proj}_{C(y_{k})}\big(w_{k+1}-A^\top\whk\big),x_{k+1}-\widehat{x}\big\rangle
		={}-\alpha_{k}\sum_{j=1}^{m}\lambda_{k,j}\left\langle\nabla f_j(y_k),x_{k+1}-\widehat{x}\right\rangle\\
		\leq {}&\alpha_{k}\sum_{j=1}^{m}\lambda_{k,j}\big(f_j(\widehat{x})-f_j(x_{k+1})-\mu_j/2\nm{y_k-\widehat{x}}^2+L_j/2\nm{x_{k+1}-y_k}^2\big)\\
		\leq {}&-\alpha_{k}\min_{1\leq j\leq m}\left[f_j(x_{k+1})-f_j(z)\right]-\frac{\mu\alpha_k}{2}\nm{y_k-\widehat{x}}^2+\frac{L\alpha_k}{2}\nm{x_{k+1}-y_k}^2\\
		= {}&-\alpha_{k}\min_{1\leq j\leq m}\big[f_j(x_{k+1})-f_j(\widehat{x})+\big\langle\widehat{\xi},Ax_{k+1}-b\big\rangle\big]-\frac{\mu\alpha_k}{2}\nm{y_k-\widehat{x}}^2\\
		{}&\qquad+\frac{L\alpha_k}{2}\nm{x_{k+1}-y_k}^2+{\alpha_k}\big\langle\widehat{\xi},Ax_{k+1}-b\big\rangle.		
	\end{aligned}
	\]
	Plugging these pieces into the decomposition of $\mathbb I_2$ leads to \cref{eq:est-I1+I2} and thus completes the proof of this theorem.
\end{proof}
\subsection{Convergence rate estimate}
Let $\bar{x}\in\Omega $ be arbitrarily fixed. In what follows, we will use all the quantities defined by \cref{eq:alpha-barx}. Following the spirit of the continuous level in \cref{sec:rate-ampd}, we can derive the upper bounds of $\nm{Ax_k-b}$ and $\snm{U(x_k)}$ with respect to the sequence $\{\theta_k\}$. The final rate is given by the decay estimate of $\theta_k$.
\begin{lemma}\label{lem:bound-xk-norm}
	Let $\{x_k,v_k,\xi_k\}$ be generated by \cref{eq:imex-qp} with the step size constraint \cref{eq:cond-ak}. Then we have $\nm{x_k}\leq R(\alpha_0(\bar{x}))$ for all $k\in\mathbb N$.
\end{lemma}
\begin{proof}
	The proof is in line with that of the continuous level in \cref{lem:bd-xt} and thus we omit the details here.
\end{proof}

\begin{theorem}\label{thm:bound-fes}
Let $\{x_k,v_k,\xi_k\}$ be generated by \cref{eq:imex-qp} with the step size constraint \cref{eq:cond-ak}. Then for all $k\in\mathbb N$, we have $\nm{Ax_k-b}\leq \theta_k/\theta_0 C_0(\alpha_1(\bar{x}),1+\alpha_2(\bar{x}))$ and $\snm{U(x_k)}\leq \theta_k/\theta_0\alpha_3(\bar{x})$.
\end{theorem}
\begin{proof}
Based on \cref{lem:bound-xk-norm}, the proof is similarly with that of the continuous case in \cref{thm:Ax-b-exp,thm:est-Uxt}. 
\end{proof}

\begin{corollary}\label{lem:theta}
Let $\{x_k,v_k,\xi_k\}$ be generated by the IMEX scheme \cref{eq:imex-qp} with the step size constraint $\alpha_k^2(L\theta_k+\nm{A}^2)= \theta_k\gamma_k$. Then we have
\begin{equation}\label{eq:rate-fes-obj}
	\left\{
	\begin{aligned}
		{}&		\snm{U(x_k)}\leq \alpha_3(\bar{x}) \theta_k/\theta_0,\\
		{}& 		\nm{Ax_k-b}\leq  C_0(\alpha_1(\bar{x}),1+\alpha_2(\bar{x}))\theta_k/\theta_0,
	\end{aligned}
	\right.
\end{equation}
where $\theta_k/\theta_0$ has the decay estimate
\begin{equation}\label{eq:est-thetak}
	\frac{\theta_k}{\theta_0}\leq \min\left\{
	\frac{2\nm{A}}{\sqrt{\gamma_0\theta_0}k}+ \frac{4L\beta_0^2}{\gamma_0k^2}
	,\,
	\frac{4\beta_0^2\nm{A}^2}{\gamma_{\min}\theta_0k^2}+\exp\left(-\frac{k\ln(1+\alpha_{\max})}{2\alpha_{\max}\sqrt{L/\gamma_{\min}}}\right) \right\},
\end{equation}
with $\gamma_{\min}:=\min\{\mu,\gamma_0\},\,\gamma_{\max}:=\max\{\mu,\gamma_0\},\,\alpha_{\max}:=\sqrt{\gamma_{\max}}\big( L +\nm{A}^2\big)^{-1/2}$ and $\beta_0:=2+\sqrt{\alpha_{\max}}$.
\end{corollary}
\begin{proof}
Since the identity $\alpha_k^2(L\theta_k+\nm{A}^2)= \theta_k\gamma_k$ satisfies \cref{eq:cond-ak}, from \cref{thm:bound-fes}, it is clear that \cref{eq:rate-fes-obj} holds true. In what follows, let us verify the decay rate \cref{eq:est-thetak}.

In view of \cref{eq:thetak-gamak}, it is not hard to find that $\gamma_{\min}\leq\gamma_k\leq\gamma_{\max}$ and
\[
\frac{\gamma_{k+1}}{\gamma_k}=\frac{1+\mu\alpha_k/\gamma_k}{1+\alpha_k}\geq \frac{1}{1+\alpha_k}=\frac{\theta_{k+1}}{\theta_k} \quad \Longrightarrow \quad \gamma_k\geq \gamma_{0}\theta_k/\theta_0.
\]
Thus, it follows that
\begin{equation}\label{eq:tk}\small
	\begin{aligned}
		\theta_{k+1} - \theta_k =
		-\sqrt{\gamma_k\theta_k}\theta_{k+1}\big( L \theta_k+\nm{A}^2\big)^{-1/2}
		\leq -\sqrt{\gamma_0/\theta_0}\theta_k\theta_{k+1}\big(\sqrt{L\theta_k}+\nm{A}\big)^{-1},
	\end{aligned}
\end{equation}
which gives
\begin{equation}\label{eq:mid-tk}
	\sqrt{L\theta_0/\gamma_0}\left(\theta_k^{-1/2}-\theta_k^{1/2}\theta_{k+1}^{-1}\right) + \nm{A}\sqrt{\theta_0/\gamma_0}\left(\theta_k^{-1}-\theta_{k+1}^{-1}\right)\leq -1.
\end{equation}
Notice that $\alpha_k=\sqrt{\gamma_k\theta_k}\big( L \theta_k+\nm{A}^2\big)^{-1/2}\leq \alpha_{\max}$ and it is easy to obtain
\[\small
\theta_k^{-1/2}-\theta_k^{1/2}\theta_{k+1}^{-1}= (1+\sqrt{1+\alpha_k})\left(\theta_k^{-1/2}-\theta_{k+1}^{-1/2}\right)\geq \beta_0\left(\theta_k^{-1/2}-\theta_{k+1}^{-1/2}\right).
\]
Therefore, from \cref{eq:mid-tk} we get
\[
\sqrt{L\theta_0/\gamma_0}\beta_0\left(\theta_{k+1}^{-1/2}-\theta_k^{-1/2}\right) + \nm{A}\sqrt{\theta_0/\gamma_0}\left(\theta_{k+1}^{-1}-\theta_{k}^{-1}\right)\geq 1.
\]
Define $\phi(t):=\sqrt{L\theta_0/\gamma_0}\beta_0t^{-1/2} + \nm{A}\sqrt{\theta_0/\gamma_0}t^{-1}$ for all $t>0$. Then we have $\phi(\theta_k)\geq k+	\sqrt{L/\gamma_0}\beta_0+ \nm{A}/\sqrt{\theta_0\gamma_0}$.
Introduce $\widehat{\theta}_k =\theta_{1,k}+\theta_{2,k} $ with
\[
\theta_{1,k}: = \frac{4L\theta_0\beta_0^2}{\big(2\sqrt{L}\beta_0+\sqrt{\gamma_0}k\big)^2},\quad
\theta_{2,k}: = \frac{2\theta_0\nm{A}}{2\nm{A}+\sqrt{\gamma_0\theta_0}k}.
\]
We claim that $\phi(\theta_k)\geq \phi(\widehat{\theta}_k)$ for all $k\in\mathbb N$. Since $\phi(\cdot)$ is monotonously decreasing, we conclude that
\begin{equation}\label{eq:thetak1}
	\frac{	\theta_k}{\theta_0}\leq\frac{ \widehat{\theta}_k}{\theta_0}\leq
	\frac{2\nm{A}}{\sqrt{\gamma_0\theta_0}k}+ \frac{4L\beta_0^2}{\gamma_0k^2}.
\end{equation}

On the other hand, since $\gamma_k\geq\gamma_{\min}$, the estimate \cref{eq:tk} becomes
\[
\theta_{k+1} - \theta_k \leq
-\sqrt{\gamma_{\min}}\sqrt{\theta_k}\theta_{k+1}\big(\sqrt{L\theta_k}+\nm{A}\big)^{-1}.
\]
Similarly, using the above argument leads to
\[
\frac{\theta_k}{\theta_0}\leq
\frac{4\beta_0^2\nm{A}^2}{\gamma_{\min}\theta_0k^2}+\exp\left(-\frac{k\ln(1+\alpha_{\max})}{2\alpha_{\max}\sqrt{L/\gamma_{\min}}}\right).
\]
Combining this with \cref{eq:thetak1}, we obtain \cref{eq:est-thetak} and finish the proof.
\end{proof}

\section{Numerical Results}
\label{sec:num}
In this section, we conduct several numerical experiments to demonstrate the practical performance of \cref{algo:amd-qp}, which is denoted as AMPD-QP for short. For the step size constraint, we choose the simple one $\alpha_k^2(L\theta_k+\nm{A}^2)=\gamma_k\theta_k $, which leads to an explicit formula $\alpha_k=\sqrt{\gamma_k\theta_k}/\sqrt{L\theta_k+\nm{A}^2}$.
\subsection{Asymptotic behavior of the dynamical system}
To provide an illustrative understanding on \cref{eq:ampd} flow, we examine its asymptotic behavior by applying the discrete algorithm AMPD-QP to some simple two dimensional bi-objective problems with a single linear equality constraint, including nonconvex, convex and strongly convex objectives.

For each problem, the initial settings for AMPD-QP are $v_0=(1,1)^\top$ and $\xi_0= 1$.
The parameters $\gamma_0$ and $\theta_0$ are randomly chosen from $(0,10]$. We consider 100 samples of the initial point $x_0\in \Bbb{R}^2$ that is randomly generated from the box $[-10,10]^2$.
In \cref{fig:1,fig:1-2,fig:1-3}, we report the numerical results of \cref{eg:eg1,eg:eg2,eg:eg3}, including the iterate trajectory, the approximate Pareto front and the average residual of all samples. Here, the residual terms contain the feasibility violation $\nm{Ax_k-b}$, the objective gap $\snm{U(x_k)}$ (cf.\cref{eq:obj-gap}) and the KKT residual (cf.\cref{eq:kkt-res}). We observe that (i) the trajectories approach to the Pareto set very well, (ii) the residual terms decrease very smoothly, and (iii) for \cref{eg:eg1}, the strong convexity implies faster rate of convergence than that of the other two examples. From this, we conclude that our dynamical \cref{eq:ampd} flow possesses good efficiency and stability for finding approximate Pareto solutions.

\begin{example}\label{eg:eg1}
	This first problem is strongly convex and taken from \cite{binh1996evolution}
	\[
	\min_{x=(x_1,x_2)\in\R^2}\left\{ x_1^2 + x_2^2, (x_1-5)^2 + (x_2-5)^2\right\} \quad \st\, x_1 - x_2 = 1.
	\]
	The Pareto optimal set is $	\mathcal P=\mathcal P_w=\big\{x=(x_1,x_1-1)\in\R^2:\,1/2\leq x_1\leq 11/2\big\}$.
\end{example}
\begin{figure}[H]
	\centering
	\begin{subfigure}[t]{0.3\textwidth}
		\centering
		\includegraphics[width=\textwidth]{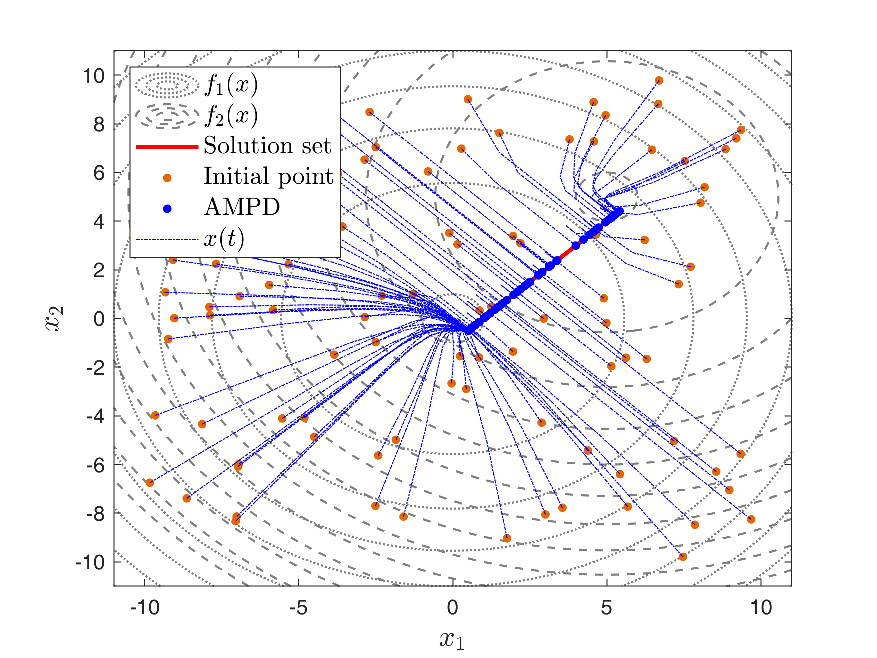}
	\end{subfigure}
	\hfill
	\begin{subfigure}[t]{0.3\textwidth}
		\centering
		\includegraphics[width=\textwidth]{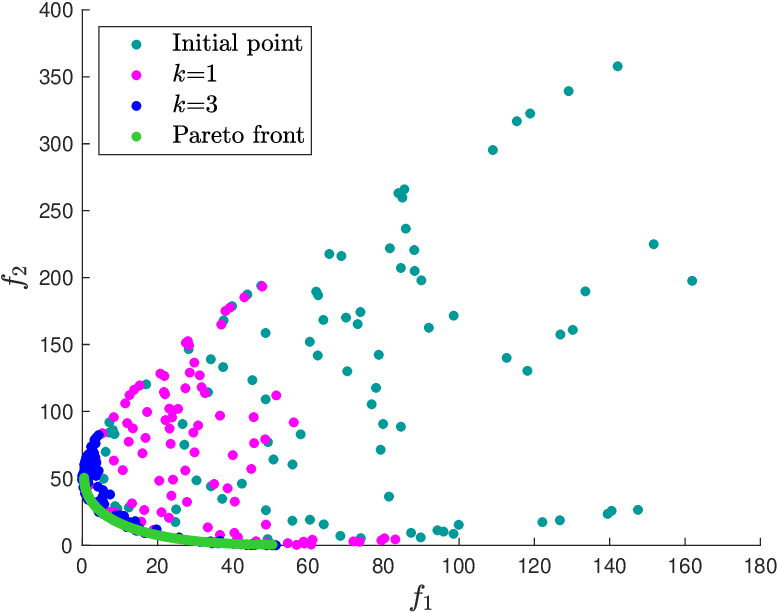}
	\end{subfigure}
	\hfill
	\begin{subfigure}[t]{0.3\textwidth}
		\centering
		\includegraphics[width=\textwidth]{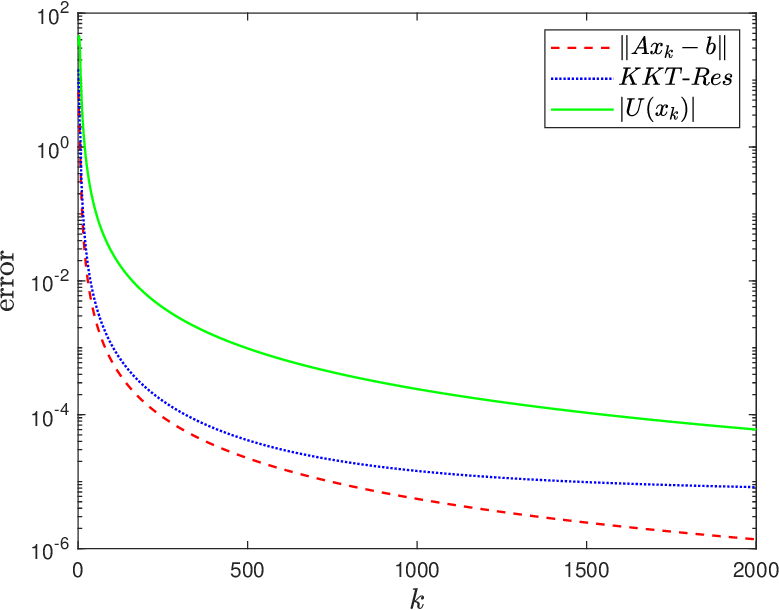}
	\end{subfigure}
	\caption{Numerical results of \cref{eg:eg1}. From left to right: the iterate trajectory, the approximate Pareto front at different step, the average residual.}
	\label{fig:1}
\end{figure}

\begin{example}\label{eg:eg2}
	The second problem reads as
	\[
	\min_{x=(x_1,x_2)\in\R^2} \left\{f_1(x),f_2(x)\right\} \quad\st\,x_1 + x_2 = 1,
	\]
	where the two convex objectives are  \cite[Section 5.2]{Sonntag2024a}
	\[
	f_i(x)=\log\sum_{j=1}^4\exp\left(\big\langle a_j^{(i)}, x\big\rangle-b_j^{(i)}\right),\quad i = 1,2,
	\]
	with the same settings for $a_j^{(i)}$ and $b_j^{(i)}$ given in \cite[Eq.(5.1)]{Sonntag2024a}. The Pareto optimal set is $	P = \left\{ x=(x_1,1-x_1) \in \R^2 : -1/2\leq x_1\leq 3/2 \right\}$.
\end{example}
\begin{figure}[H]
	\centering
	\begin{subfigure}[t]{0.3\textwidth}
		\centering
		\includegraphics[width=\textwidth]{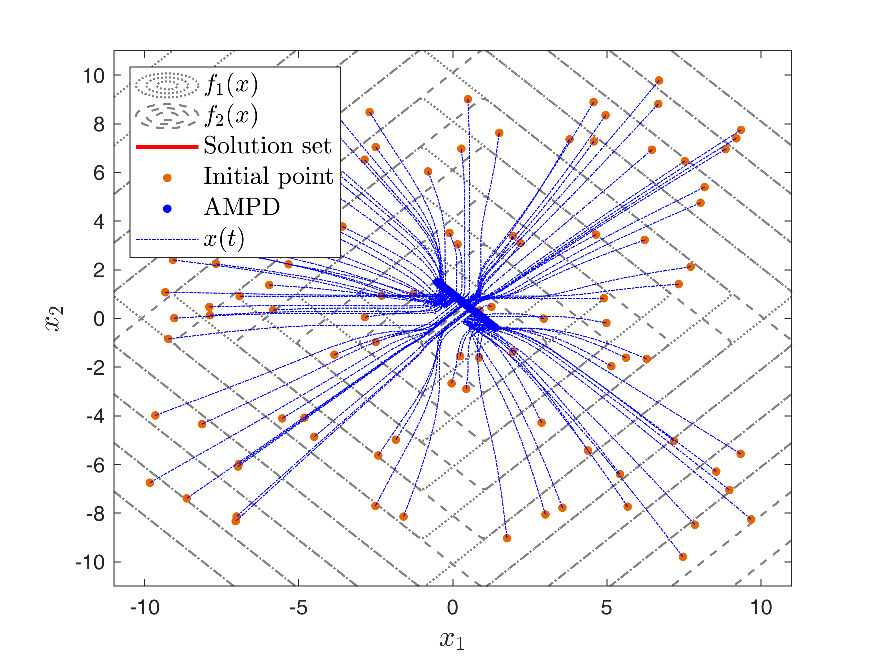}
	\end{subfigure}
	\hfill
	\begin{subfigure}[t]{0.3\textwidth}
		\centering
		\includegraphics[width=\textwidth]{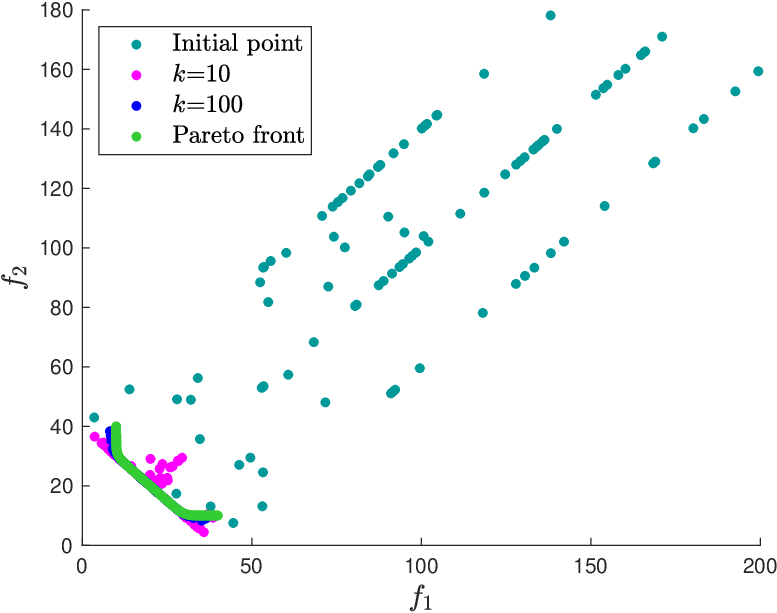}
	\end{subfigure}
	\hfill
	\begin{subfigure}[t]{0.3\textwidth}
		\centering
		\includegraphics[width=\textwidth]{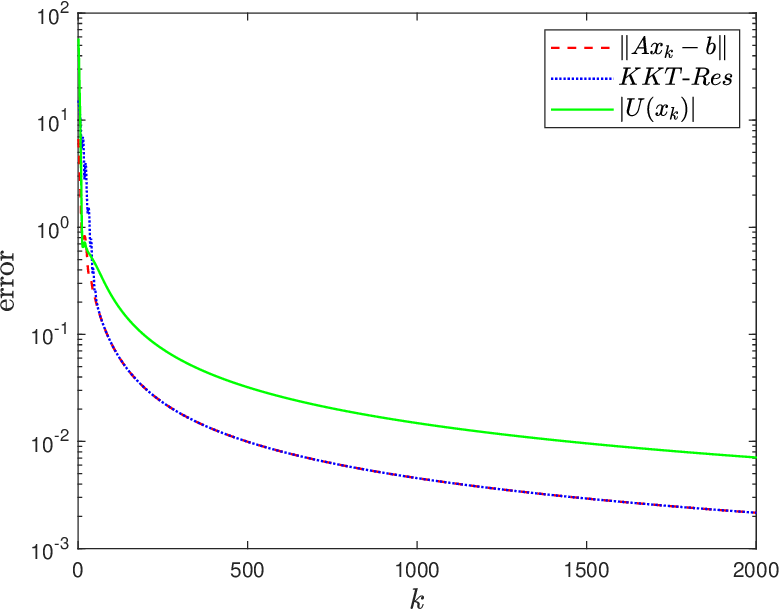}
		\end{subfigure}
		\caption{Numerical results of \cref{eg:eg2}. From left to right: the iterate trajectory, the approximate Pareto front at different step, the average residual.}
		\label{fig:1-2}
	\end{figure}
	\begin{example}\label{eg:eg3}
		The third problem is nonconvex \cite{Witting2012}
		\[
		\min_{x=(x_1,x_2)\in\R^2} \left\{ f_1(x), f_2(x)\right\} \quad \st\, x_1 - x_2 = 1,
		\]
		where
		\[
		f_{i}(x)=\frac{1}{2}\left(\sqrt{1+\snm{\dual{a,x}}^2}+\sqrt{1+\snm{\dual{b,x}}^2}+\dual{c_i,x}\right)+\lambda\exp{(-\snm{\dual{b,x}}^2)}, \quad i=1,2,
		\]
		with $a = (1,1)^\top,\,b = (1,-1)^\top,\,c_i = ((-1)^{i+1},-1)^\top$ and $\lambda=0.6$.
		The Pareto optimal set of this problem is $	P=\left\{(1/2,-1/2)\right\}$.
	\end{example}
	\begin{figure}[H]
		\centering
		\begin{subfigure}[t]{0.3\textwidth}
			\centering
			\includegraphics[width=\textwidth]{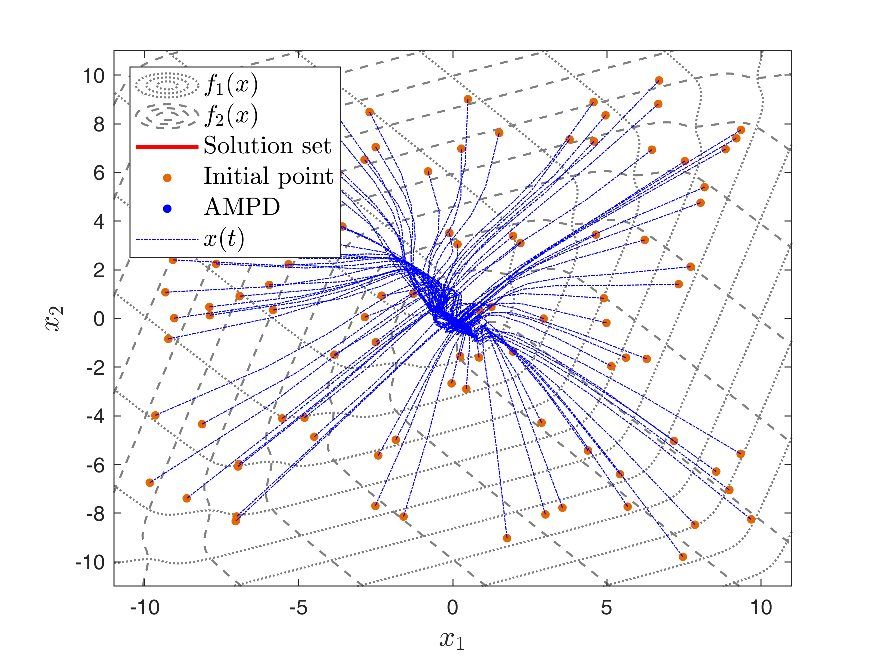}
		\end{subfigure}
		\hfill
		\begin{subfigure}[t]{0.3\textwidth}
			\centering
			\includegraphics[width=\textwidth]{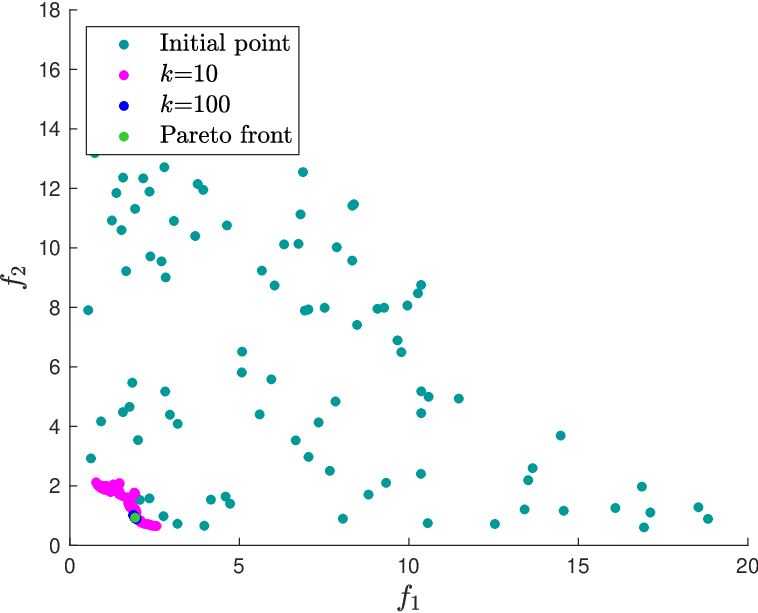}
		\end{subfigure}
		\hfill
		\begin{subfigure}[t]{0.3\textwidth}
			\centering
			\includegraphics[width=\textwidth]{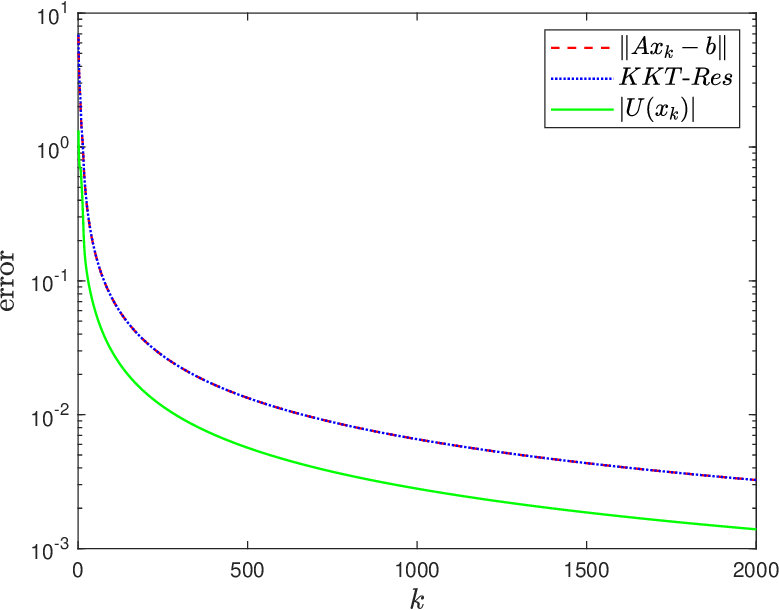}
		\end{subfigure}
		\caption{Numerical results of \cref{eg:eg3}. From left to right: the iterate trajectory, the approximate Pareto front at different step, the average residual.}
		\label{fig:1-3}
	\end{figure}

	\subsection{Comparison with existing methods}
	Now, we proceed to compare our AMPD-QP with existing methods. As noted in the introduction, aside from the augmented Lagrangian algorithm for multi-objective optimization (ALAMO) in \cite{cocchi2020} and the multiple reduced gradient (MRG) algorithm in \cite{el2018multiple,el2018new}, very few works have addressed the methods of solving \cref{eq:lcmop}. It should be emphasized, however, that MRG is based on a conventional basic variable splitting technique and is not a primal-dual type method. For this reason, we perform a series of numerical experiments to evaluate the performance of our AMPD-QP against ALAMO.

	In \cref{tab:test-prob}, we list a set of test problems from the literature, including convex and strongly convex (s.c.) objectives. For each problem, the constraint matrix $A\in\R^{m\times n}$ and the right-hand side $b\in\R^m$ are generated randomly with entries in $[-1,1]$. Further, we choose 100 starting points from the sample region given in \cref{tab:test-prob}.
	
	\begin{table}[H]
		\centering
		\caption{Test problems }
		\label{tab:test-prob}
		\begin{tabular}{ccccccc}
			\hline
			\text{Problem} & $n$ & $m$ & $r$ & \text{Sample region}& \text{Convexity} &\text{Ref.} \\
			\hline
			\text{BK1} & 2  & 2 & 1 & $[-10,10]^n$ & s.c.& \cite{binh1996evolution}\\
			\text{SPb1} & 2 & 2 & 1 & $[-200,200]^n$ &  s.c. & \cite{sefrioui2000nash}\\
			\text{TRIDIA1} & 3 & 3& 2 & $[-10,10]^n$ &  convex & \cite{toint1983test}\\
			\text{LTY1}$_{20}$&   100  & 3 & 20 & $[-1,1]^n$ &s.c. &\cite{luo_accelerated_2025}\\
			\text{LTY1}$_{50}$&   100  & 3 & 50 & $[-1,1]^n$ &s.c. &\cite{luo_accelerated_2025}\\
			\text{ZLT1}$_{20}$& 100 & 3 & 20 & $[-1,1]^n$ &s.c. &{\cite{zitzler2001spea2}}\\
			\text{ZLT1}$_{50}$& 100 & 3 & 50 & $[-1,1]^n$ &s.c. &{\cite{zitzler2001spea2}}\\
			\hline
		\end{tabular}
	\end{table}
	
	For AMPD-QP, we use the setting:
	$v_0=(1,...,1)^\top\in \R^n,\xi_0=(1,....,1)^\top\in \R^{r}$, and the parameters $\gamma_0$ and $\theta_0$ are initialized randomly in $(0,10]$.
	Since ALAMO is designed for nonlinear inequality constraint $g(x)\leq 0$, we reformulate the equality constraint $Ax=b$ as two opposite inequality constraints $Ax\leq b$ and $-Ax\leq -b$. Then for ALAMO, we use the parameter setting: $\tau_0 = 1$, $\alpha = 2$, $\mu^0 = (1,\ldots,1)^\top \in \mathbb{R}^{2r}$, and $\sigma = 0.9$. Note that in each iteration, to update the primal sequence, ALAMO has to solve an unconstrained multiobjective optimization subproblem. Following \cite{cocchi2020}, we choose the multi-objective steepest descent algorithm with Armijo-type line search \cite{fliege_steepest_2000} as an inner solver, and consider two different tolerances ${\rm tol}=10^{-4}$ and ${\rm tol}=10^{-5}$, also with the maximum number of iterations $\ell_{\max} = 8000$. For both two methods, the stopping criterion is ${\rm KKT}(x_k,\xi_k)\leq  10^{-3}$.
	
	\begin{table}[H]
		\centering
		\caption{Performances of AMPD-QP and ALAMO}
		\label{tab:test-1}
		\renewcommand{\arraystretch}{1.1}
		\begin{tabular}{l|cc|cc|cc}
			\toprule
			\multirow{3}{*}{\text{Problem}}
			& \multicolumn{2}{c|}{\multirow{2}{*}{\text{AMPD-QP}}}
			& \multicolumn{4}{c}{\text{ALAMO}} \\
			\cline{4-7}
			& \multicolumn{2}{c|}{}
			& \multicolumn{2}{c|}{$\text{tol}=10^{-4}$}
			& \multicolumn{2}{c}{$\text{tol}=10^{-5}$} \\
			\cline{2-7}
			& \text{Iter} & \text{Time}
			& \text{Iter} & \text{Time}
			& \text{Iter} & \text{Time} \\
			\hline
			\text{BK1}   & 98    & 0.13   & \textbf{87}      & \textbf{0.07}    & 42044  & 71.76 \\
			\text{SPb1}   & \textbf{382}   & \textbf{0.59}   & 955     & 0.67   & 87500  & 138.61 \\
			\text{TRIDIA1}   & 2753   & 3.81   & \textbf{768}     & \textbf{0.66}   & 58428  & 63.88 \\
			\text{LTY1}$_{20}$  & \textbf{3072}  & \textbf{4.87}  & 6752    & 18.75 & 13495   & 141.02 \\
			\text{LTY1}$_{50}$  & \textbf{4905}  & \textbf{8.83}  & 8908   & 26.79 & 18711 & 248.93 \\
			\text{ZLT1}$_{20}$ & \text{772}  & \text{1.21} & \textbf{646}   & \textbf{1.14} & 21456 & 326.50 \\
			\text{ZLT1}$_{50}$ & \textbf{1152}  & \textbf{2.10} & 2447 & 7.70 & 27513 & 458.43 \\
			\bottomrule
		\end{tabular}
	\end{table}
	
	In \cref{tab:test-1}, we report the averaged number of iterations and the CPU time (in second) of all the sample points. As we can see, ALAMO performs well for low dimension problems but is not competitive as AMPD-QP for high dimension cases. Moreover, for ALAMO, the tolerance for the inner problem has dramatic influence on the overall performance. To further validate the effectiveness of AMPD-QP, we provide more tests on the ZLT1 problem with larger $n$ and $r$ and report the numerical results in \cref{table3}, where for ALAMO, the tolerance for the inner problem is tol = $10^{-4}$. Moreover, in \cref{fig:3}, we plot the approximate Pareto fronts of AMPD-QP and ALAMO for some selected problems. It can be seen that both two methods provide good approximations but ours have better distributions for the Pareto front.
	
	\begin{table}[H]
		\centering
		\caption{Performances of AMPD-QP and ALAMO for {ZLT1}}
		\label{table3}
		\begin{tabular}{ccc cc cc}
			\toprule
			\multirow{2}{*}{${n}$} & \multirow{2}{*}{${m}$} & \multirow{2}{*}{${r}$}  & \multicolumn{2}{c}{\text{AMPD-QP}} & \multicolumn{2}{c}{\text{ALAMO}} \\
			\cmidrule(lr){4-5} \cmidrule(lr){6-7}
			& & & \text{Iter} & \text{Time} & \text{iter} & \text{time} \\
			\midrule
			100 & 3 & 50 & \textbf{1152} & \textbf{2.10}  & 2447 & 7.70  \\
			200 & 3 & 20 & \textbf{1058} & \textbf{2.06}  & 1874 & 6.85  \\
			200 & 3 & 50 & \text{1453} & \textbf{3.24}  & \textbf{1385} & 13.58 \\
			500 & 3 & 20 & \textbf{1733} & \textbf{3.93}  & 4004 & 28.37 \\
			500 & 3 & 50 & \textbf{2323} & \textbf{6.33}  & 2822 & 24.96 \\
			\bottomrule
		\end{tabular}
	\end{table}
	
	\begin{figure}[H]
		\centering
		\begin{subfigure}[t]{0.25\textwidth}
			\centering
			\includegraphics[width=\textwidth]{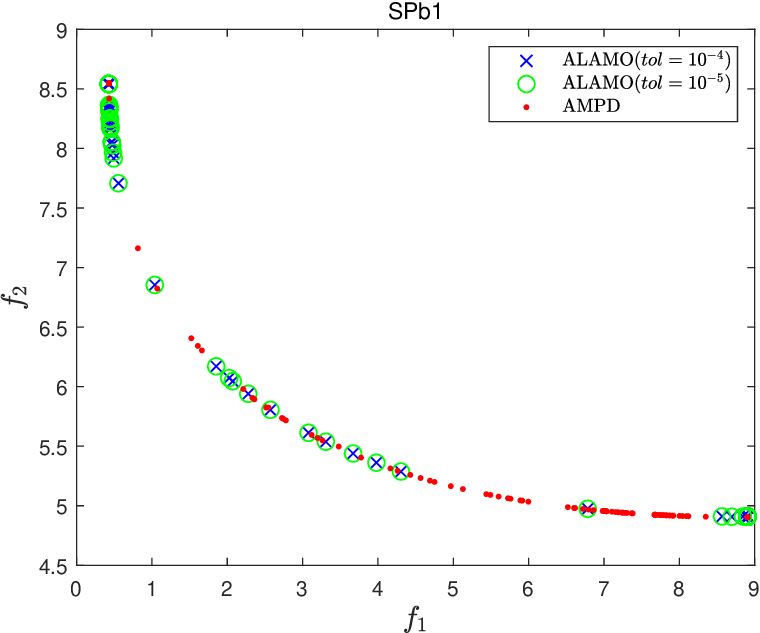}
		\end{subfigure}
		\hskip0.8cm
		\begin{subfigure}[t]{0.3\textwidth}
			\centering
			\includegraphics[width=\textwidth]{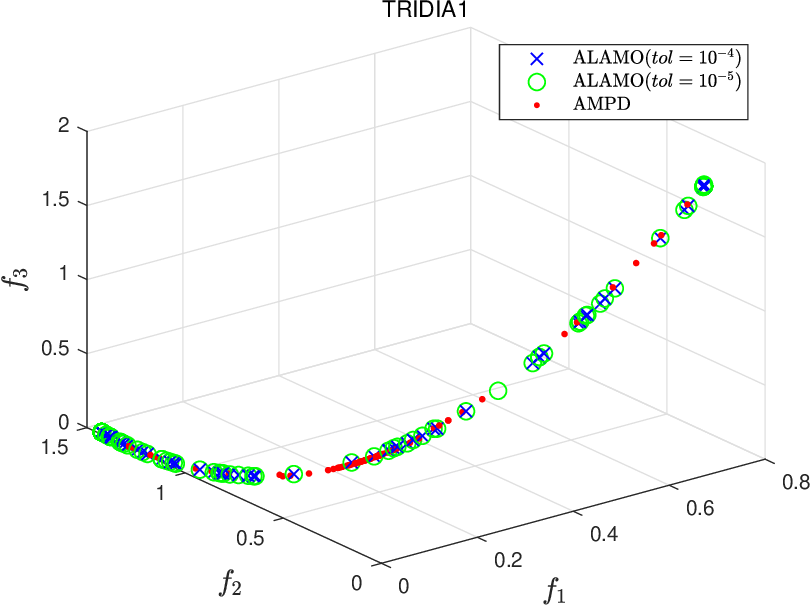}
		\end{subfigure}
		\vskip0.5cm
		\begin{subfigure}[t]{0.3\textwidth}
			\centering
			\includegraphics[width=\textwidth]{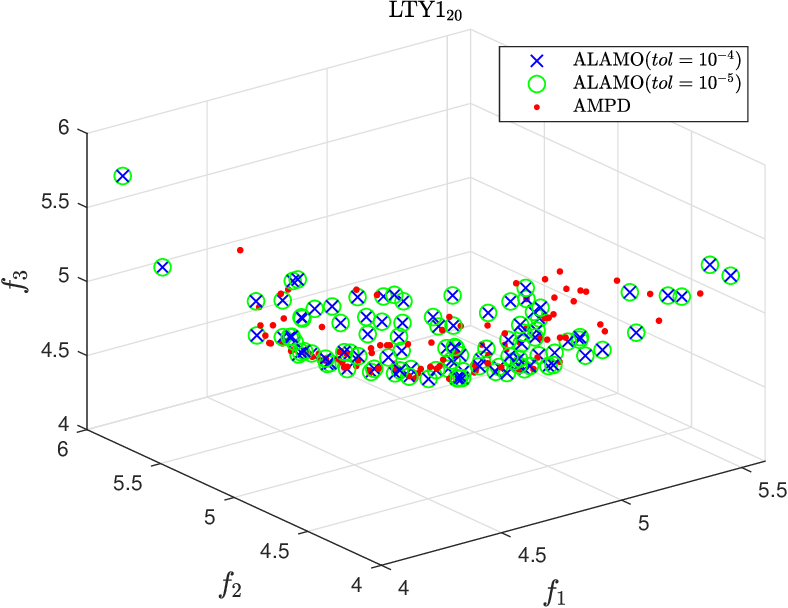}
		\end{subfigure}
		\hskip0.8cm
		\begin{subfigure}[t]{0.3\textwidth}
			\centering
			\includegraphics[width=\textwidth]{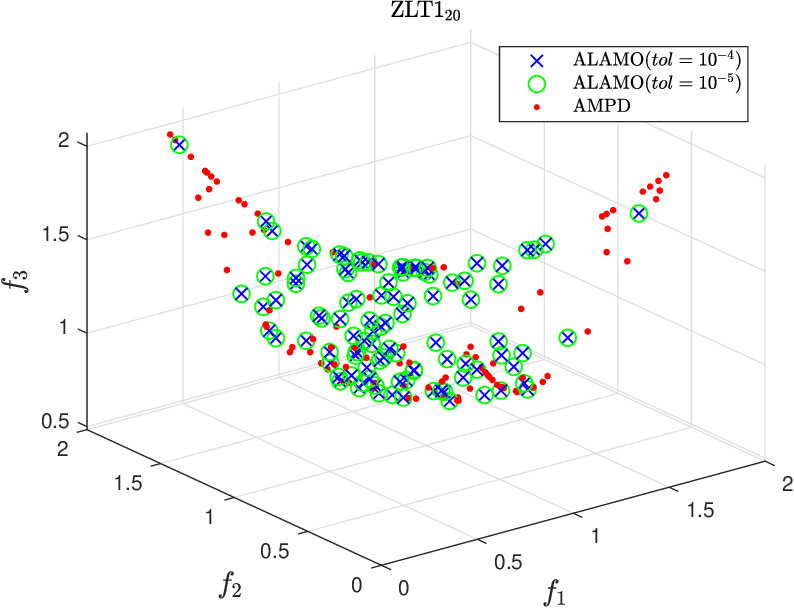}
		\end{subfigure}
		\caption{Approximate Pareto fronts of AMPD-QP and ALAMO for selected problems. }
		\label{fig:3}
	\end{figure}

	\section{Conclusion}
	\label{sec:conclu}
	In this work, we develop a novel continuous-time primal-dual framework for \cref{eq:lcmop}.  Based on a new merit function, we introduce the concept of an $\epsilon$-approximation solution to the weakly Pareto optimality. We then propose an accelerated multiobjective  primal-dual flow and establish the exponential decay via the Lyapunov analysis. In addition, we consider an implicit-explicit discretization scheme and prove that both the feasibility violation and the objective gap have the same rates $\mathcal{O}(1/k)$ and $\mathcal O(1/k^2)$ respectively for the convex case and the strongly convex case.
	
	It is worth noting that (cf.\cref{rem:x-v-xi}), the well-posedness (existence and uniqueness) of the solution to \cref{eq:ampd} requires rigorously investigations. It would also be of interest to prove the strong or weak convergence of the continuous trajectory together with its discrete counterpart. Additionally, the extension to the composite case with smooth objectives and nonsmooth objectives deserves further study. As mentioned in \cite[Section 8]{luo_accelerated_2025}, even if the multiobjective proximal gradient method \cite{tanabe2019} and the accelerated variant \cite{Tanabe2023a} have been proposed, it is still an open question to obtain the corresponding proximal gradient type methods from the continuous-time approach, as existing dynamical models \cite{attouch2015multiibjective,boct2026inertial,luo_accelerated_2025,Sonntag2024a,Sonntag2024} mainly focus on smooth objectives. We left these interesting topics as our future works.

\bibliographystyle{abbrv}

\end{document}